\documentclass{article}
\usepackage{graphicx}
\usepackage{amssymb}
\usepackage{amsthm}
\usepackage{color}
\usepackage{amsmath}
\usepackage{fullpage}
\usepackage{setspace}
\usepackage{smartref}
\usepackage{charter}
\usepackage{lscape}
\usepackage{tikz}
\usetikzlibrary {shapes}
\usetikzlibrary {arrows}
\usetikzlibrary {positioning}
\usetikzlibrary {calc}
\usetikzlibrary{fit}	

\usepackage{hyperref}
\hypersetup{backref, colorlinks=true, citecolor=blue, linkcolor=blue}
\usepackage{cleveref}
\newcommand{\rref}[1]{\hyperref[#1]{\ref*{#1}}}

\usepackage{algorithm}
\usepackage{algpseudocode}
\algnewcommand\algorithmicinput{\textbf{INPUT:}}
\algnewcommand\INPUT{\item[\algorithmicinput]}
\algnewcommand\algorithmicoutput{\textbf{OUTPUT:}}
\algnewcommand\OUTPUT{\item[\algorithmicoutput]}

\newtheorem{theorem}{Theorem}
\newtheorem{lemma}[theorem]{Lemma}

\newtheorem{corollary}[theorem]{Corollary}

\theoremstyle{definition}
\newtheorem{definition}[theorem]{Definition}

\newtheorem{example}[theorem]{Example}

\usepackage{sidecap}
\newcommand{\N}{\mathbb{N}}
\newcommand{\E}{\mathbb{E}}

\usepackage[round]{natbib}
\title{On Exchangeability in Network Models
\author{
      Steffen L. Lauritzen\thanks{Email: {\tt lauritzen@math.ku.dk}}\\
Department of Mathematical Sciences\\
University of Copenhagen\\
\and
Alessandro Rinaldo\thanks{Email: {\tt arinaldo@cmu.edu}}\\
Department of Statistics\\
Carnegie Mellon University\\
\and
Kayvan Sadeghi \thanks{Email: {\tt k.sadeghi@ucl.ac.uk}}\\
Department of Statistical Science\\
University College London\\
}
\date{\today}
}

\begin{document}

\maketitle

\begin{abstract}
We derive representation theorems for
exchangeable distributions on finite and infinite graphs  using elementary arguments based
on geometric and graph-theoretic concepts. Our results elucidate some of the key
differences, and their implications,
between statistical network models that are finitely exchangeable  and models
that define a consistent sequence of probability distributions on graphs of
increasing  size.
 \end{abstract}


\section{Introduction}

  Exchangeability is one of most basic forms of
 probabilistic invariance. When applied to probability distributions on
 graphs, it is equivalent to requiring that isomorphic
 graphs have the same probabilities. 
 Indeed, exchangeability provides the
 probabilistic underpinning to the
 theory of dense graph limits developed recently in the graph-theoretic
 literature \citep[see, e.g.,][]{diaconis:janson:08,borgs_convergent_2008,lov06}.

 In statistical network modeling,  exchangeability is a common
 simplifying assumption. However, it is typically only enforced for models on
  graphs of a given size, and not simultaneously over sequences of models on
 graphs of increasing size. This practice is born out of convenience:  it is
 much
 easier to formulate  probability
 distributions on finite as opposed to infinite graphs. However, the consequences
 of relying on this weaker assumption of
  finite exchangeability can be detrimental to the validity
 and generalizability of any statistical analysis:
 the properties of probability distributions  on graphs of different sizes that
 are finitely exchangeable need not be related to each other in any meaningful
 way (or in any way at all); see \cite{shalizi:rinaldo:13,CD:15,snijders:10}.

 In this article, we investigate the relationship between exchangeability
 of random finite graphs and exchangeability
of random infinite graphs using a combination of  simple  geometric
arguments and standard  graph-theoretic
concepts. Our work can be seen as a extension to the graph setting of the
geometric representation of finite exchangeability for random binary sequences
developed by \cite{diaconis:77}.
We make the following contributions: (1) we
formulate a finite deFinetti's theorem for random graphs that is both elementary and
rely on well known graph-theoretic quantities (namely, density homomorphisms)
only; (2) we extend this result to obtain a simple derivation of the well-known
deFinetti's
representation theorem for exchangeable distributions on (infinite)
graphs; (3) we provide novel geometric characterizations of all the finite
marginals of exchangeable distributions on finite graphs and discuss the
implications of our findings.

\noindent{\bf Related Work.} There is a vast literature on exchangeability of
random arrays, of which random graphs are a special case; see, e.g., \cite{aldous:81,aldous:85}, \cite{eagleson:weber:78}, \cite{hoover:79},
\cite{ker06}, \cite{lauritzen:08}, \cite{kallenberg:05} and \cite{silverman:76}, to name a few. Of
particular
significance is \cite{diaconis:janson:08} \citep[but see also][]{orbantz:roy:15}, which details the connections between
exchangeability of random graphs and the notion of graph limits developed in
\cite{borgs_convergent_2008} and  \cite{lov06} (see also the book \cite{lov12}).
Similarly, finite exchangeability for sequences and arrays has been thoroughly investigated: see
\cite{diaconis:77}, \cite{diaconis:freedman:81}, \cite{aldous:81} and, in particular, \cite{matus:95}; see also
\cite{volfovsky:airoldi:16}. 

 In the companion paper \cite{lauritzen:rinaldo:sadeghi:18}, we rely on tools from the theory of
 graphical models to study the Markov properties of finitely exchangeable network
 models.
The results derived there complement the ones we obtain in the present paper. We
will discuss the connection between the two papers later in  \Cref{sec:manifold}.

The article is organized in the following way. \Cref{sec:geometry} describes the
geometry of finitely exchangeable distributions on finite graphs and
exchangeable distributions on infinite graphs  and introduces the M\"{o}bius
parametrization, which we will use throughout to represent probabilities on
graphs. In \Cref{sec:densities} we provide definitions and basic results for homomorphism and isomorphism densities in order to derive a deFinetti theorem for finitely exchangeable probability distributions on graphs based on the M\"{o}bius parametrization in \Cref{sec:definetti}. In \Cref{sec:manifold} we study the manifold of dissociated and exchangeable random graphs and show that there exists dissociated and exchangeable random graphs that are not infinitely extendable.

\noindent{\bf Notation.}
For any integer $n \geq 2$ let $\mathcal{L}_n$ and
$\mathcal{U}_n$ denote
the set of simple labeled graphs and simple unlabeled graphs with node set
$[n] := \{1,\ldots,n\}$, respectively, and set $\mathcal{L} = \bigcup_{n=2}^\infty
\mathcal{L}_n$  and $\mathcal{U} = \bigcup_{n=2}^\infty \mathcal{U}_n$.
We  let $\mathcal{L}_\infty$ 
be the set of
infinite simple
labeled graphs. 
For any two graphs $G$ and $G'$ in $\mathcal{L}$, we will write $G \sim G'$ to signify
    that they are isomorphic and $[G]$ for the equivalence class of all graphs
    isomorphic to $G$.  With a slight abuse of notation, we will at times
    identify the class $[G]$ with the undirected graph representing it.
We will also identify  $\mathcal{L}_n$ with the Boolean algebra of all subsets of
the node pairs $\left\{ \{i,j\}, i \neq j \right\}$ partially ordered by inclusion by identifying each graph in
$\mathcal{L}_n$ with the binary vector $\{0,1
\}^{ { n\choose 2} }$ representing its edges. With this identification,
$\mathcal{L}_n$ indexes the coordinates of vectors in $\mathbb{R}^{2^{ {n \choose
2}}}$.
If $G$ and $H$ are in $\mathcal{L}$, we write $H \subseteq G$ if $H$ is a sub-graph (not necessarily
induced) of $G$. For integers $2 \leq m \leq n \leq \infty$ and a $G \in \mathcal{L}_n$,
$G[m]$ is the sub-graph of $G$ induced by the nodes $[m]$.

In our analysis, we will often identify a graph $G$ with its set of edges, hence ignoring isolated nodes.
The set of labeled graphs on subsets of $[n]$ without isolated nodes is denoted
by $\mathcal{I}_n$ (thus $G \in \mathcal{I}_n$ if and only if it is a subgraph of the complete graph on $[n]$ and has no isolated nodes) 
and we let $\mathcal{I} = \bigcup_{n = 2}^\infty
    \mathcal{I}_n$. Similarly, we let $\mathcal{J}_n$ denote the set of unlabeled graphs without isolated nodes and at most $n$ vertices and  $\mathcal{J} = \bigcup_{n = 2}^\infty
    \mathcal{J}_n$.

    For any finite $n \geq  2$, if $G \in \mathcal{L}_n$ and $\sigma \in \mathcal{S}_n$, the permutation
    group on $[n]$, we will let $G_\sigma$ be the graph obtained from $G$ by
    relabeling its nodes according to $\sigma$. Thus nodes $i$ and $j$ are 
    connected in $G$ if and
    only if $\sigma(i)$ and $\sigma(j)$ are connected in $G_\sigma$. Finally, we let $\mathcal{S} =
    \bigcup_n \mathcal{S}_n$ be
    the set of all finite permutations.

 \section{Probabilities on graphs: exchangeability  and geometry}
\label{sec:geometry}

We begin by introducing the notions of exchangeable distributions on networks
and illustrating their geometry properties.

For finite $n$, the  set of all  probability distributions on $\mathcal{L}_n$ can be represented
geometrically as
the probability simplex in
$\mathbb{R}^{ \mathcal{L}_n}$, denoted with $\Delta_n$. The coordinates of each vector $p \in \Delta_n$
are indexed by the graphs  in $\mathcal{L}_n$,  and the $2^{ { n \choose 2} }$
vertices of $\Delta_n$ are the unit
    masses at each $G \in \mathcal{L}_n$.

To formally define exchangeability, we first introduce the notion of {\bf marginal mapping}: for any pair of integers $2 \leq m \leq n$, this mapping is defined to be the function $\Pi_n^m
\colon \Delta_n \rightarrow \Delta_m$
mapping any probability distribution $p_n \in \Delta_n$ into the
probability distribution $\Pi^n_m(p^n) = p^m_n \in
\Delta_m$  given by
\begin{equation}\label{eq:marginal.op}
    p^m_n(H) = \sum_{ G \in \mathcal{L}_n \colon H = G[m] } p_n(G), \quad H  \in
\mathcal{L}_m.
\end{equation}
    With a slight abuse of notation we shall also think of of each $p_n \in \Delta_n$ as a measure and write $p(G[m])$ for the induced  distribution on  subgraphs:
    \[
    p(G[m]) =  p^m_n (G[m]), 
    \]
    where $p^m_n = \Pi^n_m(p)$

    \begin{definition}
    A probability distribution $p$ on $\mathcal{L}_n$ is {\bf $m$-exchangeable} when $p(G[m]) = p(G[m]_{\sigma})$ for all $\sigma \in \mathcal{S}_m$ and all $G\in \mathcal{L}_n$.
 Equivalently, $p$ is $m$-exchangeable when $p(G[m]) = p(G[m]')$ if $G[m] \sim
G[m]'$  in $\mathcal{L}_m$. If $m=n$ we say that $p$ is (finitely) {\bf exchangeable}. 
\end{definition}

We denote with $\mathcal{E}_n \subset \Delta_n$ the set of all exchangeable
distributions on $\mathcal{L}_n$. It is easy to show that exchangeable
distributions are mixtures of uniform distributions over isomorphic classes.
In fact, $\mathcal{E}_n$ is affinely isomorphic to the probability simplex in
$\mathbb{R}^{\mathcal{U}_n}$ so that $\mathcal{E}_n$ is a polytope of dimension $| \mathcal{U}_n | - 1$.
\begin{lemma}\label{lem:exchangeable.simplex}
    $\mathcal{E}_n$ is a simplex  whose vertices are the uniform probability
    distributions over isomorphic classes of $\mathcal{L}_n$. The dimension of
    $\mathcal{E}_n$ is equal to $| \mathcal{U}_n| - 1$.
\end{lemma}
 \begin{proof}
For a given class $[H] \in \mathcal{U}_n$ let $p_{[H]}$ be the probability
distribution on $\mathcal{L}_n$ corresponding to the uniform
distribution over $[H]$.
That is, for any $G \in \mathcal{L}_n$,
    \begin{equation}\label{eq:pnH}
	p_{n,[H]}(G)  = \left\{
	    \begin{array}{cc}
		\frac{1}{|[H]|} & \text{ if } G \in [H],\\
		0 & \text{ otherwise.}\\
	\end{array}
	    \right.
	\end{equation}
Then, $p_{[H]} \in \mathcal{E}_n$ for all $[H]$. The
 vectors $\{ p_{[H]}, [H] \in \mathcal{U}_n \}$ are affinely independent, because they are
supported on incomparable subsets of $\mathcal{L}_n$, regarded as a poset with
respect to the subset inclusion. Thus, their convex hull is a simplex inside
$\mathcal{E}_n$. We will show that this simplex in fact coincides with
$\mathcal{E}_n$. 
 Let
 $p$  be any point in $\mathcal{E}_n$. By exchangeability, for any $[H] \in
 \mathcal{U}_n$, the value of $p$ at each
 of the coordinates
 indexed by the graphs in the isomorphism class $[H]$ is the same.
 Thus,
\[
    p = \sum_{[H] \in \mathcal{U}_n} v_{[H]} p_{[H]},
\]
for some sequence $\{ v_{[H]}, [H] \in \mathcal{U}_n \}$ of non-negative
numbers.
Since $\sum_{ \{ G \in \mathcal{L}_n\} }
p(G) = 1$,  it follows that $\sum_{[G]} v_{[G]} = 1$ and, therefore, that $p$ is in the
convex hull of  $\{ p_{[H]}, [H] \in \mathcal{U}_n \}$. Furthermore, since the
convex hull of the vectors $\{ p_{[H]}, [H] \in \mathcal{U}_n \}$ is a simplex,
the sequence $\{ v_{[H]}, [H] \in \mathcal{U}_n \}$  is unique.
\end{proof}

\subsubsection*{Consistent Models and Exchangeability}

A highly desirable property of a probability distribution for network data of a
given size, say $m$, is that the distribution be realized as the marginal
of probability distributions over networks of larger sizes $n$ for all $n >
m$. We refer to this property as probabilistic consistency.

\begin{definition}
A sequence $\{ p_n\}_{n=2}^{\infty}$
of probability distributions
such that $p_n \in \Delta_n$ for all $n$
is {\bf consistent} if
\begin{equation}\label{eq:consistent}
p_m = \Pi_n^m p_n, \quad \forall\; 2 \leq m \leq n.
\end{equation}
\end{definition}

If a probability distribution over networks of a given size is not part of a consistent sequence, then its
properties may not be related in any meaningful way to the properties of any probability
distribution over networks of
different sizes.

When applied to an exchangeable distribution in $\mathcal{E}_n$, the
marginal mapping $\Pi^m_n$ always  yields
an exchangeable distribution in $\mathcal{E}_m$. However, $\Pi^m_n$ is
not surjective:  there are
exchangeable distributions in $\mathcal{E}_m$ that cannot be obtained
as marginals of any exchangeable distribution on $\mathcal{E}_n$, for all
$n > m>3$.
 We formally state this fact in the next result and illustrate it in \Cref{ex:no.marginal}.

\begin{lemma}\label{lem:finite.extendability}
    For all integers $4 \leq m < n_1 < n_2$, it holds that
    $\Pi^{m}_{n_2}(\mathcal{E}_{n_2}) \subsetneq \Pi^{m}_{n_1}(\mathcal{E}_{n_1})
    \subsetneq  \mathcal{E}_m $.
\end{lemma}

\begin{proof}
We will make use of the following graph-theoretic result from \cite{akiyama:etal:79}:
\begin{lemma}\label{lem:induce-subgr}
Let $H$ be a graph with $n$ nodes. We then have the following:
\begin{enumerate}
\item All the induced subgraphs of $H$ with a fixed but arbitrary number of $m$ nodes, where $2\leq m\leq n-2$, are isomorphic if and only if $H$ is a complete or empty graph.
\item All the induced subgraphs of $H$ with $n-1$ nodes are isomorphic if and only if $H$ is a {\bf node-transitive graph}, that is for any two nodes $v_1$ and $v_2$, there is some automorphism $t$ such that $t(v_1)=v_2$.
\end{enumerate}
\end{lemma}
We will first show that $ \Pi^m_n(\mathcal{E}_n)\subsetneq
\mathcal{E}_m$ for all $4 \leq m < n$
The linearity of the marginal mapping $\Pi^m_n$, $n>m$, implies that
$\Pi^m_n(\mathcal{E}_n)$ is a polytope whose vertices are contained in the image under $\Pi^m_n$
of the vertices of $\mathcal{E}_n$. 
Thus, consider a vertex of $\mathcal{E}_n$, which, by Lemma \ref{lem:exchangeable.simplex}, can be represented by an undirected
graph on $n$ nodes, say $U$. Such a vertex is mapped by $\Pi^m_n$ into a distribution giving positive probabilities to only
induced subgraphs of $U$ of size $m$. This is a convex combination of uniform
distributions over the labeled version of each of the induced subgraphs, which
are vertices of $\mathcal{E}_m$. Hence
\begin{equation}\label{eq:inclusion}
    \Pi^m_n(\mathcal{E}_n)\subseteq
\mathcal{E}_m.
\end{equation}
We now show that $\Pi^m_n(\mathcal{E}_n)$ is a strict subset of  $\mathcal{E}_m$. To see this, notice that a vertex of $\mathcal{E}_n$ is mapped into a vertex of $\mathcal{E}_m$ if and only if it corresponds to the uniform distribution over isomorphic graphs in $\mathcal{G}_n$ such that all induced subgraphs obtained by removing any set of $n-m$ nodes are isomorphic.
By Lemma \ref{lem:induce-subgr}, if $n -m = 1$ this condition is satisfied by all
node-transitive graphs and if $n-m > 2$
only by the empty and complete graphs. This proves that the inclusion
\eqref{eq:inclusion} is strict.

We will now prove that $\Pi^{m}_{n_2}(\mathcal{E}_{n_2}) \subsetneq
\Pi^{m}_{n_1}(\mathcal{E}_{n_1})$, for all integers $4
\leq n_1<n_2$.  Since $ \Pi^{m}_{n_2}(\mathcal{E}_{n_2})  =\Pi^m_{n_1}
\left( \Pi^{n_1}_{n_2}(\mathcal{E}_{n_2}) \right) $ and, as we just saw, $\Pi^{n_1}_{n_2}(\mathcal{E}_{n_2}) \subsetneq \mathcal{E}_{n_1}$, it holds that $\Pi^{m}_{n_2}(\mathcal{E}_{n_2}) \subseteq
\Pi^{m}_{n_1}(\mathcal{E}_{n_1})$. Thus, we only need to verify that the previous inclusion is strict.
This, in turn, will follow if we exhibit a vertex $p$ of $\mathcal{E}_{n_1}$ that (i) is not in the image under $\Pi^{n_1}_{n_2}$ of $\mathcal{E}_{n_2}$ and (ii) such that $\Pi^{m}_{n_1}(p)$ is a vertex of $\Pi^{m}_{n_1}(\mathcal{E}_{n_1})$.
We choose $p$ to be the uniform distribution over graphs in $\mathcal{L}_{n_1}$ that are isomorphic to the node-disjoint union of the complete graph on $n_1 - 1$ nodes and one isolated node.  By definition, this is a vertex of $\mathcal{E}_{n_1}$ and, by \Cref{lem:induce-subgr}, is not in $\Pi^{n_1}_{n_2}(\mathcal{E}_{n_2})$, since it does not belong to the image under $\Pi_{n_2}^{n_1}$ of the vertices of $\mathcal{E}_{n_2}$. Next, $\Pi^{m}_{n_1}(p)$ obviously belongs to $\mathcal{E}_m$ and can be expressed as the mixture $\frac{m}{n_1} p' + \frac{n_1 - m}{n_1} p''$. Here, $p'$ the uniform distribution over all graphs in $\mathcal{L}_m$ that are isomorphic to the node-disjoint union of a complete graph on $m-1$ nodes and one isolated node, and $p''$ is the point mass at the complete graph in $\mathcal{L}_m$. In particular, $\Pi^{m}_{n_1}(p)$  must be a vertex of $\Pi^m_{n_1}(\mathcal{E}_{n_1})$. To see this, the node-disjoint union of a complete graph on $m-1$ nodes and one isolated node is not node-transitive, and, by \Cref{lem:induce-subgr}, cannot be  a vertex of $\Pi^m_{n_1}(\mathcal{E}_{n_1})$.  Since $\Pi^{m}_{n_1}(p)$ is the only point in $\Pi^m_{n_1}(\mathcal{E}_{n_1})$ that has such a mixture representation and is the image of a vertex of $\mathcal{E}_{n_1}$, the claim follows. Thus $p$ satisfies both properties (i) and (ii) and the proof is complete.
\end{proof}

\begin{example}\label{ex:no.marginal}
Let $p_5$ the uniform distribution on $\mathcal{L}_5$ that
assigns probability $1/15$ over all graphs isomorphic to the union of a four
cycle and an isolated node. Then, its image $p_5^4$ in $\mathcal{E}_4$  under the marginal mapping is the convex combinations of two
vertices of $\mathcal{E}_4$:  the uniform distributions over
the $3$ graphs isomorphic to the 4-cycle and the uniform distribution over
the $12$ graphs isomorphic to the union of a 2-star and an isolated node. The
weights of this mixtures are $3/15$ and $12/15$, respectively.
On the other hand,
$P^4_5(\mathcal{E}_5)$ does not contain those two vertices of $\mathcal{E}_4$,
verifying that $P^4_5(\mathcal{E}_5)$ is a strict subset of $\mathcal{E}_4$.
Furthermore, the point $p_5^4$ happens to be a vertex of $P^4_5(\mathcal{E}_5)$. To see this, it is
enough to observe that, for each of the other 33 unlabeled graphs in
$\mathcal{U}_5$, the set of induced subgraphs obtained by removing any
one node is different than (and is never contained in) the set consisting of the 4-cycle and the union of a 2-star and an isolated
node. Thus, it is not possible to represent $p_4^5$ as a convex combination of
the marginals of uniform distributions over isomorphic graphs on $5$ nodes. Since
each vertex of $P^4_5(\mathcal{E}_5)$ is the image of some vertex of
$\mathcal{E}_5$, the claim follows.
\end{example}

    The previous example has led us to conjecture that, for $4 \leq m < n$,
    each vertex of $\mathcal{E}_n$ is mapped into a vertex of $\Pi_n^m(\mathcal{E}_
    n)$. When $m=3$ and $n=4$ this is clearly not true, since in this case it is easy to see that $\Pi^3_4(\mathcal{E}_4) =
    \mathcal{E}_3$, which explains the requirement that $m \geq 4$ in \Cref{lem:finite.extendability}.
For example, the uniform
    distribution over graphs in $\mathcal{L}_4$ isomorphic to the $3$-path is not a vertex of $\Pi^3_4(\mathcal{E}_4)$.

\Cref{lem:finite.extendability} implies that a sequence of finitely
exchangeable probability distributions on graphs need not be consistent.
 That is, if one poses a finitely
exchangeable distribution $p_n$ on $\mathcal{L}_n$, while all its marginals will be
exchangeable, there is no guarantee {\it a priori\/} that $p_n$ can be realized as the marginal
of any exchangeable distribution on larger graphs.

In order for a sequence of finitely exchangeable distribution on graphs to be
consistent, finite exchangeability needs to be replaced by the stronger notion
of exchangeability, which we define next. We remark that, though our definition
may appear different from the classic definition of row and column exchangeability of symmetric
random binary
arrays as in, e.g., \cite{silverman:76}, \cite{aldous:81}, \cite{diaconis:freedman:81} and \cite{lauritzen:08}, it is in fact equivalent. Below, we will use the symbol ``$\stackrel{d}{=}$'' to denote identity in distribution.

To present exchangeability, we first notice that the definition of marginal map can be extended in a straightforward manner to distributions on $\mathcal{L}_\infty$\footnote{Unlike the set $\mathcal{L}_n$, which is finite for each $n$, $\mathcal{L}_\infty$ is uncountable. Viewed as as the product set  $ \{
               0,1\}^{E(K_{\infty})} $, where $E(K_{\infty})$ denotes
               all subsets of edges of an infinite complete graph,  $\mathcal{L}_\infty$  is a compact metric
           space under the product topology. Thus, Borel probability measures are well defined on $\mathcal{L}_\infty$.}: for any integer $m\geq
2$ and any probability distribution $p_{\infty}$ over $\mathcal{L}_\infty$,
$\Pi^m_\infty$ takes $p_\infty$ into
the distribution
$\Pi^m_\infty(p_\infty) = p^m_\infty$ in $\Delta_m$ given by
\[
    p^m_\infty(H) = \mathbb{P} ( G[m] = H), \quad H  \in
\mathcal{L}_m,
\]
where $G$ is the random graph in $\mathcal{L}_{\infty}$ with distribution $p_{\infty}$. By slightly abusing notation again, for any $p \in \mathcal{L}_\infty$ and $G \in \mathcal{L}_m$, we write $p(G)$ for $p^m_\infty(G)$, where  $p^m_\infty = \Pi^m_\infty(p)$.

\begin{definition}
    A probability distribution $p$ on $\mathcal{L}_{\infty}$ is
    exchangeable when $G \stackrel{d}{=} G_\sigma$, where $G$ denotes the random graph in $\mathcal{L}_\infty$ with distribution  $p$, for all $\sigma
    \in \mathcal{S}$.
Equivalently, $p$ is  exchangeable when, for any pair $G$ and $G'$ of
isomorphic graphs in $\mathcal{L}$, $p(G) = p(G')$.
\end{definition}

It follows that all the finite marginals $\{ p_n \}_{n=1,2,\ldots}$ of an
exchangeable distribution define a consistent sequence of finitely exchangeable
distributions. Conversely, by the Kolmogorov-Bochner extension theorem \citep[see,
e.g.][]{rao:71},  the existence of a
consistent family of finitely exchangeable distribution will guarantee the
existence of an exchangeable distribution on $\mathcal{L}_\infty$.

We let $\mathcal{E}_\infty$ denote the set of all exchangeable distributions on
$\mathcal{L}_\infty$.
$\mathcal{E}_\infty$ can be identified with a compact subset of
$[0,1]^{\mathcal{I}}$
 and is a Bauer simplex; see \eqref{eq:zinfty} and Section~\ref{sec:harmonic} below.
One of our goals in this article is to describe the relationship between finite
exchangeability and exchangeability in the present setting. In particular,  we
seek a geometric characterization of the subset of $\mathcal{E}_m$ given by
\[
    \lim_{n \rightarrow \infty} \Pi_n^m(\mathcal{E}_n) = \bigcap_n
    \Pi_n^m(\mathcal{E}_n) =
    \Pi^m_\infty(\mathcal{E}_\infty ),
\]
which, in light of \Cref{lem:finite.extendability}, is  a well-defined closed
set. We
provide a partial solution in
\Cref{sec:manifold}.

\subsubsection*{The M\"{o}bius parametrization}
Though canonical, the parametrization corresponding to the set 
$\mathcal{E}_n$ is not the most convenient. As we will see,
exchangeable distributions on graphs are better represented using marginal,
as opposed to joint, probabilities.  We will refer to this parametrization as the
{\bf M\"{o}bius parametrization,} which we describe next. We take note that this
is not a novel parametrization: it is simply the adaptation to the network setting of the
well-known representation of multivariate binary distributions by means of the
M\"{o}bius inversion formula.

     Let $M_n$ be the square matrix of
     dimension $2^{ { n \choose 2} } $ with entries indexed by graphs in $\mathcal{L}_n$ and given by
                 \begin{equation}\label{eq:M}
             M_n(F,G) = 1(F \subseteq G), \quad F,G \in \mathcal{L}_n.
                            \end{equation}
                Then, $M_n$ has full rank  \cite[see,
                e.g.,][]{ST:11} and its inverse
 has entries
            \begin{equation}\label{eq:mobius.inverse}
                M^{-1}_n(F,G) = (-1)^{| G \setminus F|}1(F \subseteq
                G), \quad F,G \in \mathcal{L}_n,
                \end{equation}
                        where, for $F \subseteq G$, $| G
                        \setminus F|$ is the number of edges $G$
                        has in excess of $F$.
                    Borrowing the terminology from \cite{drt08},
                    we define the {\bf M\"{o}bius simplex} to be the
                    set
                    \[
                        \Delta^{M}_n = \{ M_n p, \, p \in \Delta_n
                    \}.
                \]
 As
                    $\Delta^M_n$ and $\Delta_n$ are in
                    one-to-one correspondence with each other,
                    $\Delta^M_n$ is a valid parametrization of
                    all the probability distribution on
                    $\mathcal{L}_n$.

The form and probabilistic interpretation of $\Delta^M_n$ are of course quite
different from those of $\Delta_n$. Indeed, we will index the coordinates of the point in $\Delta^M_n$ by the elements of
$\mathcal{I}_n$, which we recall is the set of labeled graphs on subsets of $[n]$
without isolated nodes. For a $p \in \Delta_n$ and a $z \in \Delta^M_n$ with $z
= M_np$, the value of $z$ at any such graph $F$ is just the
{\bf marginal probability}
that $F$ is a subgraph of a random graph drawn from $p$. That is,
\begin{equation}\label{eq:zH}
    z(F) = \sum_{   H \in \mathcal{L}_n \colon F \subseteq H }p(H) =
    \mathbb{P}\left( F \subseteq G \right), \quad F \in \mathcal{I}_n,
\end{equation}
where $G$ is a random graph in $\mathcal{L}_n$ with distribution $p$. In
particular, the values of $z$ conform to the partial order on $\mathcal{I}_n$: $z(F) \geq z(F')$  if $F \subseteq F'$.
Since the
M\"{o}bius transform \eqref{eq:M} is linear and invertible,  $\Delta^M_n$ is
also a polytope (in fact a simplex),
whose vertices are the image by $M_n$ of the vertices of $\Delta_n$:
$2^{ {n \choose 2}} - 1$ vertices indexed by all non-empty graphs in
$\mathcal{I}_n$
and the vector $1 \in \mathbb{R}^{ 2^{ {n
    \choose 2} }  }$, the M\"{o}bius transform of the point mass on the complete
    graph.

The M\"{o}bius parametrization enjoys the following property, referred
to as {\bf backward compatibility}. For $m \leq n$,  let $p_n \in \Delta_n$
and $p^m_n = \Pi^m_n p_n \in \Delta_m$ be its marginal, and let $z_n$ and $z^m_n$ denote their
M\"{o}bius transforms. Then,
\begin{equation}\label{eq:backward.compatibiility}
    z^m_n(F) = z_n(F), \quad \forall F \in \mathcal{I}_m.
\end{equation}
Backward compatibility is a direct consequence of the fact the M\"{o}bius
parameters are marginal probabilities.

The image of the simplex $\mathcal{E}_n$ of finitely exchangeable probability
distributions by the M\"{o}bius transform is also a polytope (in fact, a simplex) of
the same dimension, denoted by $\mathcal{E}_n^M$, whose vertices are the M\"{o}bius transform of the
vertices of $\mathcal{E}_n$. Clearly, the vertices of $\mathcal{E}_n^M$ can also be
indexed by $\mathcal{I}_n$.
By exchangeability and \Cref{eq:zH}, for each $z \in \mathcal{E}_n^M$, $z(F) = z(F')$ whenever $F
\sim F'$ in $\mathcal{I}_n$. In fact, using \Cref{eq:mobius.inverse}, 
$\mathcal{E}_n^M$ is defined geometrically by these linear
constraints, the linear constraint that $z(\emptyset) = 1$ (where $\emptyset$
signifies the empty graph), and  the facet defining inequalities
\begin{equation}\label{eq:facet.exch.mobius.simplex}
 \sum_{ F' \in \mathcal{I}_n \colon F
    \subseteq F'} (-1)^{|F' \setminus F
    |} z(F') \geq 0, \quad \forall F \in \mathcal{I}_n, F \neq \emptyset.
\end{equation}

\section{Homomorphism and isomorphism densities}
\label{sec:densities}

Next, we will recall some graph-theoretic quantities that play a key role in
our derivations. It is not a coincidence that these very same  quantities are
also used in the
theory of graph limits. See, e.g., \cite{lov12}.

Let $G \in \mathcal{L}_n$ and $F \in \mathcal{L}_m$, where $m
\leq n$ (this last requirement is not technically necessary; however we will
assume it throughout). The \textbf{homomorphism density} of $H$ in $G$ is
    \begin{equation}\label{eq:hom}
    t_{\mathrm{hom}}(F,G) = \frac{\mathrm{hom}(F,G)}{n^m},
\end{equation}
and is equal to the fraction of all mappings from $[m]$ into $[n]$ that define a
    homomorphism (adjacency preserving mapping) between $F$ and $G$.
  Density homomorphisms are {\bf multiplicative:}
    \begin{equation}\label{eq:multiplicative}
t_{\mathrm{hom}}(F_1F_2,G) =
    t_{\mathrm{hom}}(F_1,G) t_{\mathrm{hom}}(F_2,G),
    \end{equation}
where $F_1F_2$ is the
node-disjoint union of $F_1$ and $F_2$. As we will see, this
is the graph-theoretical counterpart to a fundamental
 probabilistic property known in the literature on
exchangeable arrays as the dissociated property. 

    A related concept is that  of the \textbf{injective
    homomorphism density} of $F$ in $G$,
    \begin{equation}\label{eq:inj}
    t_{\mathrm{inj}}(F,G) = \frac{\mathrm{inj}(F,G)}{ (n)_m},
\end{equation}
    where $(n)_m =  n!/(n-m)!$,
    and ${\mathrm{inj}}(F,G) $ is the number of injective mappings from $[m]$ into $[n]$ that define a
    homomorphism between $F$ and $G$.

    \noindent {\bf Remark.} If $F$ has isolated nodes, the values of both $t_{\mathrm{hom}}(F,G)$ and
    $t_{\mathrm{inj}}(F,G)$ do not change if $F$ is replaced by the smaller
    sub-graphs induced by the nodes of positive degree.  Therefore, there is no
    loss of generality in assuming that $F \in \mathcal{I}_n$ when dealing
    with the quantities in \eqref{eq:hom} and \eqref{eq:inj}.

    In a similar manner, for $G \in \mathcal{L}_n$ and $F \in \mathcal{L}_m$, we define the \textbf{isomorphism density} of $F$ in $G$
    as
    \[
    t_{\mathrm{iso}}(F,G) = \frac{\mathrm{iso}(F,G)}{n^m},
    \]
    where $\mathrm{iso}(F,G)$ is the number of maps from $[m]$ into
    $[n]$ that preserve both adjacency and non-adjacency, i.e., such that the
    induced subgraph of $G$ is isomorphic to $F$.
    Finally, let
    \begin{equation}\label{eq:ind.graph}
    t_{\mathrm{ind}}(F,G) = \frac{\mathrm{ind}(F,G)}{(n)_m}
\end{equation}
    be the injective isomorphism density, where $\mathrm{ind}(F,G)$ is the
    number of injective mappings from $[m]$ into
    $[n]$ that preserve both adjacency and non-adjacency, i.e. the number of
    isomorphisms from $F$ into induced subgraphs of $G$ with $m$ nodes.

     The next two results, whose proofs are straightforward and therefore
     omitted, provide  a more statistically transparent
   interpretation of homomorphism and isomorphism densities.
The difference between injective and non-injective densities is precisely the
difference between sampling with and without replacement.
    \begin{lemma}
	\label{lem:noinj}
    	Fix a $G \in \mathcal{L}_n$.  Let $(U_1,\ldots,U_m)$ be independent random
    variables uniformly distributed over $[n]$ and $H$ be the random graph on $[m]$ where
    $i \sim j$ in $H$ if and only if $U_i \sim U_j$ in $G$.
    Then, for any graph $F \in \mathcal{L}_m$,
    \[
	t_{\mathrm{hom}}(F,G) = \mathbb{P} \left( F \subseteq H \right) \quad
	\text{and} \quad 	t_{\mathrm{iso}}(F,G) = \mathbb{P} \left( F = H
	\right).
	\]

	Let  $(U'_1,\ldots,U'_m)$ be the
	sequence of random variables describing  the
    outcomes of $m$ draws without replacements of $n$ labelled equiprobable balls. Let $H$ be a random
    graph on $m$ nodes such that $i \sim j$ if and only if $U'_i \sim U'_j$ in $G$.
     Then, for any graph $F \in \mathcal{L}_m$,
    \[
	t_{\mathrm{inj}}(F,G) = \mathbb{P}' \left( F \subseteq H \right) \quad
	\text{and} \quad 	t_{\mathrm{ind}}(F,G) = \mathbb{P}' \left( F =
	H \right).
	\]
    \end{lemma}

Using the above representation we immediately obtain the following well-known
bound on the difference
    between subgraph densities arising from injective and non-injective
    mappings. These bounds will also be
used in the proof of \Cref{thm:main}.
    \begin{lemma}\label{lem:combinatorial}
 Let $G \in \mathcal{L}_n$ and $m \leq n$. For any $A \subseteq \mathcal{L}_m$,
 set
 \begin{equation}\label{eq:iso.A}
	    t_{\mathrm{ind}}(A,G) = \sum_{F \in A}t_{\mathrm{ind}}(F,G) \quad \text{and}
	    \quad  t_{\mathrm{iso}}(A,G) = \sum_{F \in A}t_{\mathrm{iso}}(F,G).
	\end{equation}
Then,
 \begin{equation}\label{eq:iso.A.bound}
     \sup_{A \subseteq \mathcal{L}_m}  \big| t_{\mathrm{iso}}(A,G) - t_{\mathrm{ind}}(A,G)
\big|\leq  1- \frac{(n)_m}{n^m}.
\end{equation}
As a result,
\begin{equation}\label{eq:hom.bound}
\sup_{F \in \mathcal{L}_m}   \big|  t_{\mathrm{hom}}(F,G) - t_{\mathrm{inj}}(F,G)
  \big| \leq  1- \frac{(n)_m}{n^m}.
 \end{equation}
 \label{lem:bound}
\end{lemma}
\begin{proof}
    It is enough to prove \eqref{eq:iso.A.bound}, since \eqref{eq:hom.bound} clearly follows from it.
    By \Cref{lem:noinj},  \eqref{eq:iso.A.bound} can be established by a well known bound on
    the total variation distance between a sample with and without replacement:
    see, e.g., \cite{freedman:77}.
    Here we give an alternative proof based on sub-graph densities, and inspired by the arguments used in \cite{matus:95}.
    Let $\mathrm{notinjiso}(F,G)$ denote the number of non-injective mappings from $[k]$ into
    $[n]$ that define isomorphisms between $F$ and $G$, so that
    $\mathrm{iso}(F,G) = \mathrm{notinjiso}(F,G) + \mathrm{ind}(F,G)$. For any $A
    \subseteq \mathcal{A}_m$, let $\mathrm{notinjiso}(A,G) = \sum_{F \in A}
    \mathrm{notinjiso}(F,G)$. By \Cref{lem:noinj},
both $t_{\mathrm{ind}}(A,G) $ and $t_{\mathrm{iso}}(A,G) $ are probabilities
and, therefore, take values in $[0,1]$.
    Thus,
\begin{align*}
    t_{\mathrm{iso}}(A,G) -  t_{\mathrm{ind}}(A,G)	& =
    \frac{\mathrm{notinjiso}(A,G) +
	\mathrm{ind}(A,G)}{n^k}  - \frac{\mathrm{ind}(A,G)}{(n)_k} \\
	 & =  \frac{\mathrm{notinjiso}(A,G)}{n^k}  -
	 \frac{\mathrm{ind}(A,G)}{(n)_k}
\left( 1 - \frac{(n)_k}{n^k} \right).
    \end{align*}
Since, trivially, $ 0 \leq \mathrm{notinjiso}(A,G) \leq n^k - (n)_k$, we obtain
that $0 \leq  \frac{\mathrm{notinjiso}(A,G)}{n^k}
\leq  \left( 1 - \frac{(n)_k}{n^k} \right)$. Using both bounds in the
previous display yields that
\[
- t_{\mathrm{ind}}(A,G) \left( 1 - \frac{(n)_k}{n^k} \right) \leq
t_{\mathrm{iso}}(A,G) -  t_{\mathrm{ind}}(A,G) \leq \left( 1 -
t_{\mathrm{ind}}(A,G) \right)	\left( 1 - \frac{(n)_k}{n^k} \right).
    \]
The claimed bound follows since $0 \leq
t_{\mathrm{ind}}(A,G) \leq 1$. 
\end{proof}

\noindent {\bf Remark.} The above bound can be weakened to the simpler bound $ {m \choose 2}/n$. See also Lemma 2.1 in
\cite{lov06}.

The value of $t_{\circ}(F,G)$ remains unchanged if one or
both of its arguments $F$ and $G$ are replaced by isomorphic graphs $F' \sim F$ and $G' \sim
G$, where $t_{\circ}$ is any of the densities
introduced above.
Thus, these graph densities remain well defined if one or both of their
arguments belong to $\mathcal{U}$. The next result uses this fact to establish
 a correspondence between injective densities and the concepts introduced in
 \Cref{sec:geometry}.
 It will be used in the proof of \Cref{thm:main}.

\begin{lemma}\label{lem:zn.density}
   Let $p_n$ be the vertex  of $\mathcal{E}_n$ corresponding to the
   uniform distribution over  the class $[G]$ and $p_n^m$ its
   image under the marginal mapping $\Pi^m_n$, where $2 \leq m < n$. Let
   $z^m_n$ be the
   M\"{o}bius transform of $p^m_n$.  Then, for all $F \in \mathcal{I}_m$,
   $z^m_n(F) = t_{\mathrm{inj}}(F,[G])
   $ and, for all $F \in \mathcal{L}_m$,
   $p^m_n(F) =   t_{\mathrm{ind}}(F,[G])$.
\end{lemma}
\begin{proof}
We will give a proof only for the identity involving the injective homomorphism density, since the same
arguments apply to the one involving the injective isomorphism density.
Let $[H] \in \mathcal{U}_n$ be a given isomorphism class in $\mathcal{L}_n$ and
$p_{[H]}$ the point in $\mathcal{E}_n$ corresponding to the uniform
distribution over $[H]$. For a given $F \in \mathcal{I}_m$, let $B =
\{ G \in \mathcal{L}_n \colon F \subseteq G \}$. For any $\sigma \in
\mathcal{S}_n$, let $\sigma^{-1}(B) = \{ G \in \mathcal{L}_n \colon F \subseteq
G_\sigma \}$, that is, the set of $G$ such that $G_{\sigma} \in B$. With a
slight abuse of notation we write $p_{[H]}(B) = \sum_{ G \in B} p_{[H]}(G)$. Then,
\begin{align*}
    z^m_{[H]}(F) & = z_{[H]}(F)\\
& = \sum_{G \in \mathcal{L}_n} 1(F \subseteq G) p_{[H]}(G) \nonumber \\
    & = p_{[H]}(B)\\
    & = \frac{1}{n!} \sum_{\sigma \in
	\mathcal{S}_n} p_{[H]}( \sigma^{-1}(B))\\
	& = \frac{1}{n!} \sum_{G \in [H]} \sum_{\sigma \in
	    \mathcal{S}_n} 1(F \subseteq G_\sigma) p_{[H]}(G)\\
	    & = \sum_{G \in [H]} \frac{ \left| \{ \sigma  \in
		\mathcal{S}_n \colon
	    F \subseteq G_\sigma \} \right|}{n!} p_{[H]}(G)\\
	    & = t_{\mathrm{inj}}(F,[H]) \sum_{G \in [H]} p_{[H]}(G)\\
	    & = t_{\mathrm{inj}}(F,[H]).
    \end{align*}
    The first identity follows from the backward compatibility of the M\"{o}bius
    transform, the fourth identity follows from exchangeability and the last
    identity uses the facts that, for any $\sigma \in \mathcal{S}_n$, $G_\sigma \in
    [H]$ if and only if $G \in [H]$ and that $t_{\mathrm{inj}}(F,\cdot)$ is
    constant over $[H]$ (with  the common value denoted as
    $t_{\mathrm{inj}}(F,[H])$).
\end{proof}

\section{deFinetti theorems for exchangeable distributions on graphs}\label{sec:definetti}
\subsection{A finite deFinetti theorem}

	In this section we will use the sub-graph densities  introduced in
	\Cref{sec:densities} to derive a  deFinetti theorem for
	finitely exchangeable probability distributions on graphs based on the
	M\"{o}bius parametrization. The results show that the
	M\"{o}bius parameters  of a finitely exchangeable distribution on
	$\mathcal{L}_m$ that
	extends to a finitely exchangeable distribution on $\mathcal{L}_n$,
	where $2 \leq m < n$, are
	the expected values of the injective density homomorphisms. These are
 approximated uniformly well
	by  the expected values of the density homomorphisms with the
	approximation error of order $O({m^2}/{n})$. The proof is a
	simple application of \Cref{lem:zn.density} and
	\Cref{lem:combinatorial} and is the graph-theoretical counterpart of the
	proof of Theorem 1 in \cite{matus:95}.

    \begin{theorem}[{\bf deFinetti's theorem for finitely exchangeable
    	distributions on graphs}]
	\label{thm:main}
	Let $p_n \in \mathcal{E}_n$ and $p^m_n = \Pi^m_n p_n$  where $m \leq n$.
	Let $z^m_n$ be the corresponding M\"{o}bius
	transform of $p^m_n$.
	Then, for any subgraph $F \in \mathcal{I}_m$,
\begin{equation}\label{eq:exch.second}
 z^m_{n}(F) = \sum_{G \in \mathcal{L}_n}
				     t_{\mathrm{inj } }(F,G) p_n(G)
				 \end{equation}
    	and
\begin{equation}\label{eq:exch.second.2}
    \max_{ F  \in \mathcal{I}_m} \Big| z^m_n(F) - \sum_{G \in \mathcal{L}_n}
	    t_{\mathrm{hom}}(F,G) p_n(G) \Big| \leq 1 - \frac{(n)_m}{n^m}.
\end{equation}
\end{theorem}
\begin{proof}
	The proof relies on Lemma \ref{lem:bound} and can be regarded as extension  to the network
	setting of the
	geometric arguments used in  \cite{diaconis:77}. See also \cite{diaconis:freedman:81} and
	\cite{ker06}, and in particular, \cite{matus:95}.

			Far any isomorphism  class
			$[H] \in \mathcal{U}_n$,
	 and any
		$F \in \mathcal{I}_m$,
    \begin{equation}\label{eq:inj.pointmass}
 t_{\mathrm{inj}}(F,G) =  t_{\mathrm{inj}}(F,G'), \quad \forall G, G' \in
 [H].
 \end{equation}
 As before, $t_{\mathrm{inj } }(F,[H])$ denotes
				the common value of $t_{\mathrm{inj } }(F,G)$
				for all graphs $G\in [H]$.

	Next,  let $p_{n,[H]}$
    be the finitely exchangeable probability distribution on $\mathcal{L}_n$
    corresponding to the uniform distribution over the isomorphic class $[H]$,
    as described in \Cref{eq:pnH}.
	Since $p_{n,[H]}$ is a vertex of $\mathcal{E}_n$ by
	\Cref{lem:exchangeable.simplex}, its M\"{o}bius transform  $z_{n,[H]}$ is a vertex of
	$\mathcal{E}_n^M$. Then, because $\mathcal{E}_n^M$ is a simplex, any point in
	$z_n \in \mathcal{E}_n^M$ can be written as
				\[
z_n =  \sum_{[H] \in \mathcal{U}_n} w_{[H]} z_{n,[H]},
				    \]
				    for a unique sequence of non-negative
				    numbers $\{
				    	w_{[H]}, [H] \in \mathcal{U}_n\}$ such
					that $\sum_{[H] \in \mathcal{U}_n}
					w_{[H]} = 1$.

					Let  $p_{n,[H]}^m = \Pi^m_n
					p_{n,[H]}$ be
     the probability distribution over
    $\mathcal{L}_m$ obtained by marginalizing over $p_{n,[H]}$ and  $z^m_{n,[H]}$ be
			 its   M\"{o}bius
			    transform.  
    Then, for any $F \in \mathcal{I}_m$,
    \begin{equation}\label{eq:inj.expand}
				z^m_{n,[H]}(F) = t_{\mathrm{inj } }(F,[H]) =
				\sum_{G \in [H]} \frac{1}{[H]} t_{\mathrm{inj
				} }(F,G) ,
			    \end{equation}
				where the first identity follows from
				\Cref{lem:zn.density}.

   As a result, for any $F \in \mathcal{I}_m$,
    \begin{align*}
z_n^m(F) & = \sum_{[H] \in \mathcal{U}_n} w_{[H]} z^m_{n,[H]}(F) \\
& = \sum_{[H] \in
    \mathcal{U}_n} w_{[H]} t_{\mathrm{inj}}(F,[H])\\
    & =  \sum_{[H] \in \mathcal{U}_n}  w_{[H]} \left( \sum_{G \in [H]}
    \frac{1}{|[H]|} 	t_{\mathrm{inj}}(F,G) \right)  \\
    & = \sum_{[H] \in \mathcal{U}_n}  w_{[H]} \left( \sum_{G \in [H]} p_{n,[H]}(G)
    t_{\mathrm{inj}}(F,G) \right)  \\
    & = \sum_{G \in \mathcal{G}_n} t_{\mathrm{inj}}(F,G) \left( \sum_{[H] \in \mathcal{U}_n} w_{[H]}
    p_{n,[H]}(G) \right) \\
    & = \sum_{G \in \mathcal{G}_n}   t_{\mathrm{inj}}(F,G)p_n(G),
    \end{align*}
    where the first, second and fourth identities follow from the linearity of
marginal operation, \Cref{eq:inj.expand}, and \Cref{eq:pnH},
respectively. Thus, \Cref{eq:exch.second} follows.

Using the previous identity, for a given $F \in \mathcal{I}_m$ of size, say, $k$,
\begin{align*}
\left| z_n^m(F) -\sum_{G \in \mathcal{G}_n}
t_{\mathrm{hom}}(F,G) p_n(G)  \right| & = \left| \sum_{G \in \mathcal{G}_n}
t_{\mathrm{inj}}(F,G) p_n(G)  -\sum_{G \in \mathcal{G}_n}
t_{\mathrm{hom}}(F,G) p_n(G)  \right| \\
& \leq \sum_{G \in \mathcal{G}_n} \left|
t_{\mathrm{inj}}(F,G) - t_{\mathrm{hom}}(F,G)
\right| p_n(G)\\
& \leq 1 - \frac{(n)_k}{n^k},
    \end{align*}
   where the last inequality is due to \Cref{lem:bound}.
   \Cref{eq:exch.second.2} is established by noting that
\begin{equation}\label{eq:bound}
				    1 - \frac{(n)_k}{n^k} \leq 1 -
				    \frac{(n)_m}{n^m} ,
				\end{equation}
				    for all integer $k < m$.
    \end{proof}	

	 The theorem further implies that, for any finitely exchangeable
	 distributions on $\mathcal{L}_n$, the marginal probabilities of
	 all its small sub-graphs are well
	       approximated by a certain mixture of
	       densities homomorphisms of such sub-graphs, with the mixing
	       measure defined over isomorphisms class in $\mathcal{L}_n$.
	       Formally, we have the following:

	       \begin{corollary}
		   \label{cor:finite.ex}
		   Assume $2 \leq m < n$. Let $p_n \in \mathcal{E}_n$ and $z_n$ be its M\"{o}bius
		   transform. Then, there exists a probability
		   distribution $\{ w_{U}, U \in
		   \mathcal{U}_n \}$ on $\mathcal{U}_n$, uniquely determined by $p_n$,  such that
		   \[
		      \left| z_n(F) - \sum_{U \in \mathcal{U}_n} w_{U}
		       t_{\mathrm{hom}}(F,U) \right| \leq 1 -
		       \frac{(n)_m}{n^m}, \quad \forall F \in \mathcal{I}_m.
		   \]
	       \end{corollary}

	 \begin{proof}
By backward compatibility \eqref{eq:backward.compatibiility}, $z^m_n(F) = z_n(F)$ for all $F \in
	 \mathcal{I}_m$. Furthermore, by \Cref{lem:exchangeable.simplex}, each
	 $p_n \in \mathcal{E}_n$ can be written as
	 \[
	     p_n = \sum_{[G] \in \mathcal{U}_n} p_{n,[G]} w_{[G]},
	 \]
	 for a unique probability distribution   $\{ w_{[G]}, [G] \in
	 \mathcal{U}_n \}$ on $\mathcal{U}_n$, where $p_{n,[G]}$ is the uniform
	 distribution over the class $[G] \in \mathcal{U}_n$. The claim follows
	 from the fact that $t_{\mathrm{hom}}(F,\cdot)$ takes on the same
	 value  $t_{\mathrm{hom}}(F,[G])$ over $[G]$ and collecting terms.
	       \end{proof}

	             By
         \Cref{lem:exchangeable.simplex}, $p_n$ being extremal is equivalent to $p_n$ being a
         uniform distribution over some isomorphism class, say $[G]$ in
         $\mathcal{L}_n$. If  $p_n$ is extremal then $w_{[G]} = 1$.
Since homomorphism densities  are multiplicative, we can use the fact that
$1-{(n)_m}/{n^m} \leq  {m \choose 2}/n$ to conclude that, for $n$
of larger order than $m^2$ and if $p_n$ is an extremal distribution on
$\mathcal{L}_n$,
\begin{equation}\label{eq:approx.dissociatedness}
z_n(F) \approx z_n(F_1)z_n(F_2),
\end{equation}
for each $F \in \mathcal{I}_m$ of the form $F = F_1F_2$, where we recall that $F_1F_2$ is the vertex-disjoint union of $F_1$ and
$F_2$. As we will show in the next section, the approximation in
\Cref{eq:approx.dissociatedness} becomes an equality if $p_n$ is embedded into a
sequence of consistent finitely exchangeable distributions  that
extend to an extremal exchangeable distribution over $\mathcal{L}_\infty$.
Furthermore, all such extremal distributions are defined by these identities.

A result
analogous to \Cref{thm:main}
 holds also for joint
	       probabilities. We have chosen to focus on marginal
	       probabilities since they are more natural in this context, as they directly lead to the key
	       approximation property of
	       \Cref{eq:approx.dissociatedness}.
	       \begin{corollary}\label{cor:main}
		   Consider the setting of \Cref{thm:main}.
	Then, for any $F \in \mathcal{L}_m$,
	\begin{equation}\label{eq:exch.first}
	    p^m_n(F) = \sum_{G \in \mathcal{L}_n}   t_{\mathrm{ind}}(F,G)p_n(G)
	\end{equation}
	    and, as a result,
	\begin{equation}\label{eq:exch.first.2}
	    \max_{ F \in \mathcal{L}_m} \Big| p^m_n(F) - \sum_{G \in \mathcal{L}_n}
	    t_{\mathrm{iso}}(F,G) p_n(G) \Big| \leq 1 - \frac{(n)_m}{n^m}.
	\end{equation}
Furthermore, letting $\tilde{p}^m_n$ the probability distribution on
	$\mathcal{L}_m$ given by
	\[
	    \tilde{p}^m_n(F) = \sum_{G \in \mathcal{L}_n}
	    t_{\mathrm{iso}}(F,G) p_n(G)  , \quad F \in \mathcal{L}_m,
	\]
	we have
	\begin{equation}\label{eq:TV}
	    d_{\mathrm{TV}} \left(   p^m_n , \tilde{p}^m_n \right)  \leq 1 - \frac{(n)_m}{n^m},
	\end{equation}
	where $d_{\mathrm{TV}}(P,Q)$ denotes the total variation distance between the
    probability distributions $P$ and $Q$. 
\end{corollary}

    \begin{proof}
	We omit the proofs of \eqref{eq:exch.first} and \eqref{eq:exch.first.2},
	since they are nearly identical to the proofs of  \eqref{eq:exch.second}
	and \eqref{eq:exch.second.2} given above.
	To prove \eqref{eq:TV}, let $A \subset \mathcal{L}_m$ and recall the
	quantities defined in \cref{eq:iso.A}:
	\[
	    t_{\mathrm{ind}}(A,G) = \sum_{F \in A}t_{\mathrm{ind}}(F,G) \quad \text{and}
	    \quad  t_{\mathrm{iso}}(A,G) = \sum_{F \in A}t_{\mathrm{iso}}(F,G).
	\]
	Notice that by \Cref{lem:bound}, $|  t_{\mathrm{ind}}(A,G) - t_{\mathrm{iso}}(A,G)  | \leq 1 -
	\frac{(n)_m}{n^m}$, for any $G \in \mathcal{L}_n$ and $A \subseteq
	\mathcal{L}_m$. Then,
\begin{align*}
    \left| \sum_{F \in A}  p_n^m(F) - \sum_{F \in A}  \tilde{p}_n^m(F) \right| & =
    \left| \sum_{F \in A}  \left( \sum_{G \in \mathcal{G}_n}
t_{\mathrm{ind}}(F,G) p_n(G) \right)  - \sum_{F \in A} \left( \sum_{G \in \mathcal{G}_n}
t_{\mathrm{iso}}(F,G) p_n(G) \right)  \right|\\
& = \left| \sum_{G \in \mathcal{L}_n}  \left( t_{\mathrm{ind}}(A,G) -
t_{\mathrm{iso}}(A,G) \right) p_n(G)   \right| \\
& \leq  \sum_{G \in \mathcal{L}_n}  \left| t_{\mathrm{ind}}(A,G) -
t_{\mathrm{iso}}(A,G) \right| p_n(G)  \\
& \leq 1 - \frac{(n)_m}{n^m}.
    \end{align*}
    Inequality \eqref{eq:TV} now follows since, by definition,
    \[
    	d_{\mathrm{TV}}(p^m_n, \tilde{p}^m_n) = \sup_{A \subseteq \mathcal{L}_m}
    	\left| \sum_{F \in A}  p_n^m(F) - \sum_{F \in A}  \tilde{p}_n^m(F)
    	\right|.
\]This completes the proof.
	 \end{proof}

    Just like in \Cref{cor:finite.ex}, we can equivalently express
    \eqref{eq:exch.first} as
    \[
	    p^m_n(F) = \sum_{U \in \mathcal{U}_n}   t_{\mathrm{ind}}(F,U)w_U,
    \]
    for a probability distribution $\{ w_U, U \in \mathcal{U}_n\}$ that is uniquely dtermined by $p_n$.

\subsection{From finite exchangeability to exchangeability}
		    Below we will strengthen the conclusions of \Cref{thm:main}
		    by further assuming that each $p_n \in \mathcal{E}_n$ is an
		    element of a sequence  $\{ p_n
		    \}_{n=2}^\infty$ of
		    finitely exchangeable distributions   that are
		    consistent, i.e. satisfy  \Cref{eq:consistent}. As noted
		    above, each such
		    sequence extends uniquely to one element in
		     the simplex $\mathcal{E}_\infty$ of
		    exchangeable probability distribution on
		    $\mathcal{L}_\infty$. Below, we will establish a deFinetti
		    type of 
		    theorem for exchangeable distributions and, along the way,
		    relate it to the theory of graph limits. This
		    connection is well known and has been elucidated in
		    \cite{diaconis:janson:08}. 

We begin by introducing a few concepts that are necessary to
represent distributions over infinite graphs and graph sequences. First, it is easy to see that any probability
		distribution on $\mathcal{E}_\infty$ admits a M\"{o}bius
		parametrization that is completely analogous to the one given
		for distributions of finite random graphs and are based on
		marginal probabilities of finite subgraphs without isolated
		nodes.
		   In detail, for a point $p_\infty \in \mathcal{E}_\infty$
		    we will write 
		    \begin{equation}\label{eq:zinfty}
		    	z = z(p_\infty) = (z(F), F \in \mathcal{I})
			\in [0,1]^{\mathcal{I}},
		    \end{equation}
		    for the sequence 
		    of M\"{o}bius parameters given by
		    \begin{equation}\label{eq:z.infty}
			z(F) =  
			\mathbb{P}( F \subset G) \quad
			 F \in \mathcal{I}, 
		    \end{equation}
		    where $G$ is
 an infinite random graph with distribution $p_\infty$ and we recall that
 $\mathcal{I}$ is the set of all finite graphs without isolated nodes. In particular,
	the M\"{o}bius parametrization $\left( z(F), F \in \mathcal{I} \right) \subset
	[0,1]^{\mathcal{I}}$ of an exchangeable
	distribution on $\mathcal{L}_\infty$ satisfies the properties that $z(\emptyset) = 1$,
 $z(F) = z(F')$
	if $F \sim F'$ and \Cref{eq:facet.exch.mobius.simplex} holds for all
	$n$.

We will also require the notion of graph limits: see \cite{lov06},
\cite{borgs_convergent_2008}, \cite{lov12}.
    Following \cite{diaconis:janson:08}, we let $\mathcal{U}_\infty$ be
    the collection of all sequences 
    \begin{equation}\label{eq:zinfty2}
    	(x_F, F \in \mathcal{I}) \in [0,1]^{\mathcal{I}}
    \end{equation}
    of the form
    \[
	x_F = \lim_n t_{\mathrm{hom}}(F,U_n)
    \]
    for some sequence $\{ U_n \}_n$ of unlabeled graphs, with $U_n \in
    \mathcal{U}_n$ for all $n$. The set $\mathcal{U}_\infty$ consists of all
    possible limits of sequences of unlabeled graphs, according to the
    definition of graph limit of \cite{lov06} and
\cite{borgs_convergent_2008}.  Intuitively, one can think of each $U \in
\mathcal{U}_\infty$ as an ``infinite unlabeled graph". Indeed, notice the similarity between \eqref{eq:zinfty2}
and \eqref{eq:zinfty}.
    In order to emphasize the role of density
    homomorphisms in this definition 
    we will
    write  $t_{\mathrm{hom}}(F,U)$ for the element of the sequence $U \in \mathcal{U}_\infty$ corresponding to
     $F \in \mathcal{I}$. Notice that if $F$ and
     $H$ are isomorphic then $t_{\mathrm{hom}}(F,U) = t_{\mathrm{hom}}(H,U)$,
     for all $U \in \mathcal{U}_\infty$.
     The set 
     $\mathcal{U}_\infty$ is a compact subset of the compact metric space
     $[0,1]^\mathcal{I}$ when endowed with the metric
			   \[
			       d\left( x,y \right) = \sum_{i=1}^\infty
			       \frac{1}{2^i}|x_{F_i} - y_{F_i}|,
			       \]
where $F_1,F_2,\ldots$ is an enumeration of all the graphs in
$\mathcal{I}$.

	    The next result provides a representation of the
	    M\"{o}bius parameters of the probability distributions in
	    $\mathcal{E}_\infty$ as expected density homomorphism of
	    all the graphs in $\mathcal{I}$. In addition, the M\"{o}bius parameters of the extremal distributions in
	    $\mathcal{E}_\infty$ satisfy a defining set by polynomial
	    equations given below in \eqref{eq:dissociatedness}.
	\begin{theorem}[{\bf deFinetti Theorem for exchangeable random networks}]
	    \label{thm:definetti}
	    The M\"{o}bius parameters
	    corresponding to  the probability distribution $p_\infty \in \mathcal{E}_\infty$
	    are given by
	    \begin{equation}\label{eq:main2}
    z(F) = \lim_n \mathbb{E} \left[  t_{\mathrm{hom}}(F,G[n]) \right] =  \mathbb{E} \left[  t_{\mathrm{hom}}(F,U)\right] , \quad \forall F \in \mathcal{I}.
\end{equation}
where the expectation is with respect to the distribution of a random variable
$U$ taking values in $\mathcal{U}_\infty$. 
A distribution on $\mathcal{L}_\infty$ is extremal in $\mathcal{E}_\infty$ if and only if its
M\"{o}bius parameters satisfy the conditions
	\begin{equation}\label{eq:dissociatedness}
 	z(F) =
    z(F_1)z(F_2)
\end{equation}
for all $F \in \mathcal{I}$ with $F = F_1F_2$.
Furthermore, there exists one deterministic graph limit $U \in \mathcal{U}_\infty$  such that
	\begin{equation}\label{eq:dissociatedness2}
z(F)   = t_{\mathrm{hom}}(F,U) = \lim_n  t_{\mathrm{hom}}(F,G[n]), \quad \forall F \in \mathcal{I},
\end{equation}
where the limit exists almost surely. 
	\end{theorem}

    \begin{proof}
		   Let $G \in
		   \mathcal{L}_\infty$ be an exchangeable infinite labeled
		   graph. Then, 
		   for
		   a fixed  $m \geq 2$ and each  $n  \geq m$, \Cref{thm:main} yields that
		   \begin{equation}\label{eq:squeeze}
		       \Big| \mathbb{P} \left( G[m]\supset F \right) -
		       \mathbb{E}\left[ t_{\mathrm{hom}}(F,G[n]) \right] \Big| \leq 1
		    - \frac{(n)_m}{n^m}, \quad \forall F  \in \mathcal{I}_m.
		       \end{equation}

		     Let $\{ U_n\}_n \subset \mathcal{U}$ be a sequence of random unlabeled graphs such that, for each $n$, $U_n$ represents the isomorphism class of $G[n]$. Then, for each $n$, $t_{\mathrm{hom}}(F,G[n]) \stackrel{d}{=} t_{\mathrm{hom}}(F,U_n)$, where we recall that ``$\stackrel{d}{=}$'' denotes identity in distribution, and, as result  $\mathbb{E}\left[ t_{\mathrm{hom}}(F,G[n]) \right] = \mathbb{E}\left[ t_{\mathrm{hom}}(F,U_n) \right] $.
		       Taking the limit in $n$, \eqref{eq:squeeze} implies that
		       \[
 \mathbb{P} \left( G[m] \supset F \right) = \lim_n \mathbb{E}\left[ t_{\mathrm{hom}}(F,U_n) \right],
			   \]
			   for all $F \in \mathcal{I}_m$ and all $m \in
			   \mathbb{N}$.
			   Thus, by Theorem 3.1 in \cite{diaconis:janson:08}, there exists a random element $U \in \mathcal{U}_\infty$ such that 
			   \[
			       \mathbb{P} \left( G \supset F \right) =
			       \mathbb{E}\left[ t_{\mathrm{hom}}(F,U) \right], \quad \forall F \in
		       \mathcal{I},
			       \]
			       where $\mathbb{E}\left[ t_{\mathrm{hom}}(F,U) \right] = \lim_n \mathbb{E}\left[ t_{\mathrm{hom}}(F,U_n) \right] = \lim_n \mathbb{E}\left[ t_{\mathrm{hom}}(F,G[n]) \right]$.
			       Thus,  \Cref{eq:main2} is proved. 

			       \noindent {\bf Remarks} 
\begin{enumerate}
\item In fact, in the notation, of \cite{diaconis:janson:08}  $U_n$ converges in distribution to $U$, both viewed as elements of the space $\overline{\mathcal{U}}$, and, by Theorem 5.3 therein, such a random $U$ is unique.
\item Furthermore, invoking again Theorem 3.1 in \cite{diaconis:janson:08}, we can conclude  that $t_{\mathrm{hom}}(F,G[n])$ converges in distribution for each $F \in \mathcal{I}$.
\item Alternatively, we may prove the claim using standard arguments from the theory of weak convergence of probability measures; see, e.g., the proof of Theorem 4 in \cite{dia77}. Indeed, for each $n$ we  let $\mu_n$ be the probability distribution of $U_n$ defined over the compact metric space $\overline{\mathcal{U}}$. Then, there exists a subsequence $\{ \mu_{n_i} \}_i$ that converges weakly to a probability  measure $\mu$ over the same space, which we may define to be the distribution of $U$. Since, for each fixed $F \in \mathcal{I}$, $t_{\mathrm{hom}}(F, \cdot)$ is a bounded and continuous function over $\overline{\mathcal{U}}$ \citep[see, e.g.,][]{lov12} the result follows.   

\end{enumerate}

     To show  \eqref{eq:dissociatedness}, we will rely on the following result of \cite{diaconis:janson:08}.
\begin{lemma}[\cite{diaconis:janson:08}, Corollary 5.4]\label{lem:extreme}
    There is a one-to-one correspondence between the extreme points of the set
    $\mathcal{E}_\infty$ and the set $\mathcal{U}_\infty$, given by
    \begin{equation}\label{eq:eq:extreme}
	t_{\mathrm{hom}}(F,U) = z(F), \quad \forall F \in \mathcal{I},
    \end{equation}
    where $U
    \in \mathcal{U}_\infty$ and  $z(F)$ is the value of M\"{o}bius parameter at
    $F$ for the
    corresponding $p_\infty$ (see
    \ref{eq:z.infty}).
\end{lemma}

We can now prove \eqref{eq:dissociatedness}.
Assume  that  $p_\infty$ is extremal in $\mathcal{E}_\infty$ with M\"{o}bius parameters $\{z(F), F \in
    \mathcal{I}\}$.
Let $U \in \mathcal{U}_\infty$ its corresponding
sequence. Then there exists a  sequence $\{ U_n \}_n
\subset \mathcal{U}$ of unlabeled graphs  with $U_n \in \mathcal{U}_n$ for all
$n$ such that $t(F,U) = \lim_n t_{\mathrm{hom}}(F,U_n)$, for any $F \in \mathcal{I}$. 
Consider any pair of node disjoint graphs $F_1$ and $F_2$ in
    $\mathcal{I}$. Without loss of generality, we may take the nodesets of $F_1$
    and $F_2$ to be $[m_1]$ and $\{m_1+1,\ldots,m_1 + m_2\}$, respectively.
Using \cref{lem:extreme},
\[
z(F_1F_2) = t(F_1,F_2,U) =  \lim_n t_{\mathrm{hom}}(F_1F_2,U_n) = \lim_n t_{\mathrm{hom}}(F_1,U_n)t_{\mathrm{hom}}(F_2,U_n) = z(F_1)z(F_2),
\]  
where the third identity follows from the multiplicative property of density
homomorphisms, which holds for all $n \geq m_1 +
m_2$; see \eqref{eq:multiplicative}.
The same argument applies to any pair of node-disjoint graphs $F_1$ and $F_2$
in $\mathcal{I}$, and  \eqref{eq:dissociatedness} follows.
Now suppose that \eqref{eq:dissociatedness} holds. Using  \Cref{eq:main2}, for
any pair of node-disjoint isomorphic graphs $F_1$ and $F_2$  in $\mathcal{I}$,
\[
z(F_1F_2) =  \mathbb{E}\left[ t_{\mathrm{hom}}(F_1F_2,U) \right] = \mathbb{E}\left[ t_{\mathrm{hom}}^2(F_1,
U) \right],
\]
where $U$ is the random element in $\mathcal{U}_\infty$ corresponding to the distribution
$p_\infty$.
Using \eqref{eq:dissociatedness} and the fact that $z(F_1) = z(F_2)$, we have that
\[
    \mathbb{E} \left[  t_{\mathrm{hom}}^2(F_1,U) \right] = z(F_1F_2) =
     z(F_1)z(F_2) = z^2(F_1) = \left( \mathbb{E}\left[ t_{\mathrm{hom}}(F_1,U) \right]
     \right)^2,
\]
and, therefore, that $t(F_1,U)$ is almost surely constant. Since the choice of
$F_1$ is arbitrary, we conclude that the random variable $t_{\mathrm{hom}}(F,U)$ is almost surely constant for each $F \in
\mathcal{I}$ and therefore, by definition, that $U$ is non random.  It then follows from
\Cref{lem:extreme} that the distribution of $G$ is extremal. 
Finally, since $t_{\mathrm{hom}}(F,G[n]) \stackrel{d}{=} t_{\mathrm{hom}}(F,U_n)$ for all $n$ and $F \in \mathcal{I}$ and $\{ U_n\}$ is a non-random sequence of graphs in $\mathcal{U}$ with graph limit $U$, $t_{\mathrm{hom}}(F,G[n]$ converges almost surely to $t_{\mathrm{hom}}(F,U)$. 
\end{proof}

\Cref{thm:definetti} gives a reformulation of well known results about symmetric
	binary
	exchangeable arrays \citep[see, e.g.][]{aldous:81,aldous:85,lauritzen:08,silverman:76,eagleson:weber:78,hoover:79,kallenberg:05}  and can also be directly
	linked to the theory of graph  limits, as shown in particular by
	\cite{diaconis:janson:08} \citep[see also][Chapter 11]{lov12}. Our contribution is
	a relatively simple proof that combines the finite exchangeability
	bound from \Cref{thm:main} with
	classic arguments from the theory of weak convergence of measure as
	detailed in
	\cite{diaconis:janson:08}.

The identity \eqref{eq:main2} signifies that the M\"{o}bius parameters
	of any $p_\infty \in \mathcal{E}_\infty$ can be expressed as an average
    of density homomorphisms over graph limits, while \Cref{eq:dissociatedness2}
    expresses the result that there is a one-to-one correspondence between graph
    limits and extremal distributions in $\mathcal{E}_\infty$ (this formally
    stated in \Cref{lem:extreme} above).

The proof of the Theorem also reveals that if $G$ is a random graph in $\mathcal{L}_\infty$ with an exchangeable distribution, then, as $n \rightarrow \infty$ and for each $F \in \mathcal{I}$, the sequence $t_{\mathrm{hom}}(F,G[n])$ converges in distribution; it also converges almost surely if and only if the distribution of $G$ is extremal in $\mathcal{E}_\infty$. Furthermore, we see that 
\[
z(F) = \lim_n \mathbb{E}[ t_{\mathrm{hom}}(F,G[n])], \quad F \in \mathcal{I}.
\]
 Of course, a priory, for any infinite (random or deterministic) graph in $\mathcal{L}_\infty$, the limit $\lim_n t_{\mathrm{hom}}(F,G[n])$ needs not exist.

An equivalent version of \Cref{thm:definetti}  can also be given for probability
parameters as opposed to M\"{o}bius parameters. 
However,
the parametrization of extremal distributions in
$\mathcal{E}_\infty$ by the
induced probabilities of finite graphs does not seem to satisfy
any factorization properties, such as the ones expressed in
\eqref{eq:dissociatedness} for the M\"{o}bius parameters. For this reason, we
find the
M\"{o}bius parametrization  more convenient. We refrain from providing the
details.

One of the main implications of \Cref{thm:definetti} is that, for any integer $n
\geq 2$,   if
$p_n \in \mathcal{E}_n$ is the marginal of an extremal exchangeable distribution on $\mathcal{L}_\infty$, then, by
\Cref{eq:dissociatedness},
its M\"{o}bius parameters (marginal probabilities) satisfy  the
identities
\begin{equation}\label{eq:dissociatedness.n}
z_n(F_1F_2) = z_n(F_1)z_n(F_2), \quad \forall F_1,F_2 \in \mathcal{I}_n,
\end{equation}
 i.e. the approximation \Cref{eq:approx.dissociatedness} holds exactly.
This property holding
 for all $n$ is equivalent to a well-known
 measure-theoretical property of exchangeable distributions over binary arrays,
 known as {\it dissociatedness}; see, e.g., \cite{silverman:76}. In fact,
 dissociatedness is a necessary and sufficient condition for an exchangeable
 distribution over arrays to be extremal \citep{aldous:85}. In the graph limit literature
 \citep[see, e.g.,][Chapter 11]{lov12} an equivalent formulation of
 \Cref{eq:dissociatedness.n} for all $n$ is referred to as  the {\it local}
 property of the associated sequence of distributions.

We will refer to distributions in $\mathcal{E}_n$ satisfying
 \Cref{eq:dissociatedness.n} as {\bf dissociated.} Notice that, according to
 our definition, a dissociated distribution in $\mathcal{E}_n$ needs not be the
 marginal of any dissociated or even finitely exchangeable distributions over
 larger graphs. This is in
 contrast with the classic notion of dissociatedness used in the probabilistic
 literature, which requires
 \Cref{eq:dissociatedness.n} to hold for all $n$ and therefore applies to all the
 marginals of an exchangeable distribution.  Indeed, as we will show below, there exist dissociated
 distributions in $\mathcal{E}_n$, for all $n \geq 4$, that cannot be extended to
 to any distributions on larger graphs.

Since dissociated
distributions in $\mathcal{E}_n$ contain the
M\"{o}bius parameters of the marginals of all extremal distributions in
$\mathcal{E}_\infty$, in order to understand the subset of $\mathcal{E}_n$
corresponding to  image under the marginal mapping of all exchangeable distributions
it is crucial to study dissociated distributions,
which we do next in the next Section.

\subsubsection{Connection with harmonic analysis}\label{sec:harmonic}

There is an interesting connection between \Cref{thm:definetti} and harmonic analysis on semigroups \citep{berg:etal:84,ressel:08}. More precisely, if we consider the semigroup $(\mathcal{J},+)$ of unlabeled graphs without isolated nodes, where $+$ denotes node disjoint union, the M\"obius parameters clearly satisfy
\[z(F)=\phi([F]),\]
for some function $\phi:\mathcal{J}\to \mathbb{R}_+$. A is shown in the lemma below, the function $\phi$ is \textbf{positive definite} on $(\mathcal{J},+)$, meaning that any matrix of the form
\[m_{ij}=\phi([F_i]+[F_j]),\quad  i,j=1,\ldots, n\] is positive semidefinite. 
\begin{lemma}\label{lem:positivity}Let $G$ be a random exchangeable graph with M\"obius parameters $z$ given as above. Then the function $\phi$ is bounded and positive definite on $(\mathcal{J},+)$.
\end{lemma}
\begin{proof}Clearly $\phi(\emptyset)=1$ and $\phi$ is bounded. Introduce the binary random variables $X_{ij}$ for $i\neq j\in \N$ where
$X_{ij}=1$  if $i\sim j$ in $G$ and $X_{ij}=0$ otherwise; $X$ is the (random) adjacency matrix of $G$.  Then, clearly
\[z(F)= \E\left(\prod_{ij: i\sim j \in F}X_{ij}\right).\]
So elementary calculations will verify that
\begin{eqnarray*}\sum_{u,v=1}^n c_uc_v\phi([F_u]+[F_v])&=&
\sum_{u,v=1}^n c_uc_v\E\left(\prod_{i\sim j\in F_u}X_{ij}\prod_{i\sim j\in F^*_v}X_{ij}\right)\\&=&
\E\left\{\sum_{u}c_u\prod_{i\sim j\in F_u}X_{ij}\right\}^2\geq 0
\end{eqnarray*}
where $F^*_v$ is a copy of $F_v$ which is node-disjoint from $F_u$.
This completes the proof.
\end{proof}
We note that the property in Lemma~\ref{lem:positivity} is referred to as \textbf{reflection positivity} in \cite{lov06}.

Now \cite{berg:etal:76} show that the set of bounded positive definite functions  on an Abelian semigroup is a Bauer simplex with the set of \textbf{characters} (multiplicative functions) as extreme points; this is essentially equivalent to the statement in \Cref{thm:definetti}.

\subsubsection{Connection with graphons}
The conclusions of \Cref{thm:main} can be equivalently expressed using graphons.
Indeed, paraphrasing a deep result about exchangeable arrays established by
\cite{aldous:81} and \cite{hoover:79} \citep[see also][]{kallenberg:05}, the
M\"{o}bius sequence $(z(F), F \in \mathcal{I})$ corresponding to an
extremal exchangeable distribution admits the representation
\[
    z(F) = \int_{[0,1]}  \int_{[0,1]^n} \prod_{(j,j) \in E(F)} W(x_i,x_j) d x_1 \ldots d x_n,
\]
for all $F \in \mathcal{I}_n$ and all $n \geq 2$, and some symmetric measurable
function $W \colon [0,1]^2
\rightarrow [0,1]$ (which is not uniquely defined). The same result was also
established in the context of graph limits by
\cite{borgs_convergent_2008} and \cite{lov06}, who termed the function $W$
 a graphon.
Furthermore, equation \eqref{eq:main2} takes the form
\[
    z(F) = \int_{[0,1]} \Big\{ \int_{[0,1]^n} \prod_{(j,j) \in E(F)} \phi(\alpha,
x_i,x_j) d x_1 \ldots d x_n \Big\} d \alpha,
\]
for all $F \in \mathcal{I}_n$ and all $n \geq 2$, for some measurable function
$\phi \colon [0,1]^3
\rightarrow [0,1]$ (not necessarily uniquely defined), symmetric in its last two
arguments. See, e.g., Chapter 14 in \cite{aldous:81}.

 \section{The manifold of dissociated exchangeable distributions}
 \label{sec:manifold}


 Let $\mathcal{D}_n \subset \mathcal{E}^M_n$ be the set of M\"{o}{bius parameters
     of finitely
	exchangeable dissociated distributions on $\mathcal{L}_n$. By definition,
	$\mathcal{D}_n$ is comprised of all the points in
	$\mathcal{E}^M_n$ that satisfy the system of polynomial equations
	\eqref{eq:dissociatedness.n}. Therefore, $\mathcal{D}_n$ is the intersection of
	$\mathcal{E}^M_n$ with a smooth manifold, in fact an affine variety in
	$(z(F), F \in \mathcal{I}_n)$.  For this reason, we will refer
	to $\mathcal{D}_n$ as the {\bf dissociated manifold}.

	Clearly, the image of $\mathcal{D}_n$ under the inverse M\"{o}bius
	transform is a subset of $\mathcal{E}_n$ that can be also defined by
	a system of polynomial equations in the probability parameters, though
	these relations are not as simple as the ones in \eqref{eq:dissociatedness.n}.

	The next result describes some
	of the properties of the set $\mathcal{D}_n$. In particular,  it shows
	that
 if $p_n$ is the marginal of an exchangeable, non-extremal
 distribution on $\mathcal{L}_\infty$, then its M\"{o}bius parameters (marginal
 probabilities) are mixtures of the M\"{o}bius parameters of dissociated
 distributions in $\mathcal{E}_n$. From this, we obtain a partial geometric
 characterization of the set
 $\Pi^n_\infty(\mathcal{E}_\infty)$. It should be apparent now why M\"{o}bius
 parameters (marginal probabilities) are better suited to describe
 exchangeability in our context. Recall that $M_n$ denotes the M\''o bius map defined in (\ref{eq:M}).

 \begin{lemma}\label{lem:Dn}
	   The dimension of $ \mathcal{D}_n$ is the number of unlabeled connected
	   graphs with at most $n$ nodes. If $p_\infty$ is a
	   distribution in $\mathcal{E}_\infty$ and $z = z(p_\infty)$ is as in
	   \Cref{eq:z.infty}, then
	   \[
z_n \in \left\{
		   \begin{array}{ll}
		       \mathcal{D}_n &  \text{if } p_\infty \text{ is
		       extremal}\\
		       \mathrm{convhull}(\mathcal{D}_n)  &  \text{otherwise,}
		   \end{array}
		   \right.
	   \]
	   where $z_n = ( z(F), F \in \mathcal{I}_n)$.
	   As a result, for each $n \geq 2$,
	   \[
	       P^n_\infty(\mathcal{E}_\infty) \subset
	       \left\{ M_n^{-1} z_n \colon z_n  \in
	       \mathrm{convhull}(\mathcal{D}_n)\right\}.
	   \]
	\end{lemma}

    \begin{proof}
	    That claim about the dimension of $\mathcal{D}_n$ follows from counting the number of
	   polynomial equations in \eqref{eq:dissociatedness.n} (see, e.g.,
	   \cite{drt08} and references therein for a similar calculation) and taking into
	   account the fact that $z_n(F) = z_n(F')$ for all $F \sim F'$.
	   To show the second statement, let $p_\infty \in \mathcal{E}_\infty$
	   and $z = z(p_\infty)$ be as in
	   \Cref{eq:z.infty}. Then, by \eqref{eq:main2},
	   \[
	       z_n(F) = z(F) = \mathbb{E}\left[   t_{\mathrm{hom}}(F,U)\right],
\quad
\forall F \in \mathcal{I}_n,
	   \]
	   where in the above display $U$ is a random variable taking values in $\mathcal{U}_\infty$.
	   The claim is then established after noting  that for any deterministic
	   $U \in \mathcal{U}_\infty$, it holds that $t_{\mathrm{hom}}(F,U) \in \mathcal{D}_n$ for
	   all $F \in \mathcal{I}_n$,  by
	   the second part of \Cref{thm:definetti}.
	\end{proof}

\Cref{lem:Dn} should be compared with the results  in  \cite{ELV:79}, where
it is shown, with a different language, that the set
of M\"{o}bius parameters in $\mathcal{E}^M_n$ arising from extremal
distributions in
$\mathcal{E}_\infty$ belongs to $\mathcal{D}_n$ and has a non-empty interior, of dimension equal to
the number of unlabeled connected graphs with at most $n$ nodes. The
authors further
remark that not much else is known about this set for any $n \geq 3$, including its topological
properties (though they  do show that it is path-connected).

In light of this, one may be led to conjecture that each point in the dissociated manifold
$\mathcal{D}_n$ arises as the M\"{o}bius parameter of an extremal exchangeable
distribution on $\mathcal{E}_\infty$. However, quite surprisingly, this is not
the case. The following example provides a family of strictly positive
probability distributions in $\mathcal{D}_4$ that are not extendable.
Geometrically, this set is a line segment in $\mathcal{D}_4$. Other examples of
dissociated distributions with zero entries in $\mathcal{D}_4$
that are not extendable are given in tables \ref{tab:zu} and \ref{tab:pu}, which
we discuss in the next section.

\begin{example}
Take $n=4$.  Consider the (strictly positive) probability
 distribution $p^* = \alpha p_1 + (1-\alpha) p_2$ on $\mathcal{L}_4$, where $\alpha \in (0,1)$,
 $p_1$ is the Erd\"{o}s-Renyi distribution on $\mathcal{L}_4$ with
 $p=1/2$, and $p_2$ is the distribution on  $\mathcal{L}_4$ corresponding to the mixture of the point
 mass at the empty graph, the uniform
 distribution over the $4$ graphs isomorphic to the union of a triangle and an
 isolated node, and the uniform distribution over the $3$ graphs isomorphic to the
 4-cycle, with weights $1/8$, $1/2$ and $3/8$ respectively. Both $p_1$ and $p_2$
 are finitely exchangeable, and, therefore, so is $p^*$. Furthermore,  since
 under both $p_1$ and $p_2$ the M\"{o}bius parameters corresponding to graphs
 isomorphic to an
 edge  and to the disjoint union of two edges are $1/2$ and $1/4$
 respectively,  $p^* \in
 \mathcal{D}_n$ for each $\alpha \in (0,1)$. Yet, as $p_2$
 is not contained in the image of the marginal mapping $\Pi^4_5$ over $\mathcal{L}_5$, there
 is no distribution on $\mathcal{L}_n$, $n \geq 5$,  whose marginal in
 $\mathcal{L}_4$  is $p^*$.
 Therefore, $z(p^*) \not \in \Pi^5_\infty(\mathcal{E}_\infty) $.
\end{example}
 It remains unknown  whether there exist any simple
 criteria to determine whether the inverse M\"{o}bius transform of any point in $\mathcal{D}_n$ is also in
 $\Pi^m_\infty(\mathcal{E}_\infty) $.

 \paragraph{Connections with \cite{lauritzen:rinaldo:sadeghi:18}.} In \cite{lauritzen:rinaldo:sadeghi:18}, we have also investigated exchangeable network models as
 graphical models on binary data with symmetric restrictions. There we have shown that
 distributions in $\mathcal{E}_n$ can only be compatible with
 few Markov properties, and we have identified all the possible conditional independence
 structures that such distributions  may exhibit. Furthermore, we have
 proved that the only
 non-trivial conditional independence structure that yields  a consistent
 sequence of
 finitely exchangeable probability distributions corresponds to
a certain
 bi-directed graphical model for marginal independence.
 Such a model, which can be thought of as a canonical parametric model
 encompassing all
 finitely exchangeable networks of any given size,
 belongs to the class of marginal models for binary data studied by
 \cite{drt08} \citep[see also][]{roverato:etal:13}. In particular, it is
 obtained  from enforcing  the dissociatedness constraints
 \eqref{eq:dissociatedness.n} in addition to  exchangeability.
 One of the implications
 of these results is that the image under $\Pi^m_{\infty}$ of all extremal families in
 $\mathcal{E}_\infty$ is a strict  submodel of a graphical model for marginal
 independence.
 Finally, the model can be parameterized as a curved exponential family on $\mathcal{L}_n$ with natural sufficient
 statistics given by the injective density homomorphisms and dimension equal to
 the number of connected unlabeled graphs on $n$ nodes.

\subsection{Maximum likelihood estimation}

In this section we further investigate some of the statistical properties of the
models specified by the manifold $\mathcal{D}_n$ of exchangeable and dissociated
distributions.
 We will focus on the basic problem of estimating the M\"{o}bius
parameters by maximizing the likelihood based on a sample of size one.

If $G \in \mathcal{L}_n$ is the observed network, a maximum
likelihood estimator (MLE) of the M\"{o}bius parameters under the dissociated model is
a point in
the set
\[
    \mathrm{argmax}_{ z \in \mathcal{D}_n  } \ell(G,z),
\]
where $\ell(\cdot,G)$ is the likelihood function, given by
\[
     z \in \mathcal{D}_n \mapsto \ell(z,G) = \sum_{ \{ F \in \mathcal{I}_n, \colon G
    \subseteq F\} } (-1)^{|F \setminus G
    |} z(F).
\]
Using exchangeability, we can rewrite the likelihood function as
\[
   \ell(z,G)  =  \sum_{ U \in \mathcal{U}_n } (-1)^{E(U) - E(G)
   } r_{U}(G)z(U),
\]
where for a (labeled or unlabeled) graph $G$, $E(G)$ is the number of its edges, $r_{U}(G)$ is the number of graphs in $\mathcal{L}_n$
containing $G$ as a subgraph and belonging to the isomorphism class represented
by $U$, and $z(U)$ is  the common value of the coordinates of the M\"{o}bius
parameters $z$ corresponding to the graphs in the isomorphism class represented
by $U$. See Examples 1 and 2 in \cite{lauritzen:rinaldo:sadeghi:18}.

As remarked in the previous Section, points in $\mathcal{D}_n$ correspond to the closure of
the mean-value space of a curved exponential family of probability
distributions on $\mathcal{L}_n$}.
The MLE  of the M\"{o}bius parameters may be on the boundary of
$\mathcal{D}_n$ and may
not be unique. Both cases are problematic from a statistical standpoint: the
former case implies that the probability
distribution corresponding to the MLE assigns zero probability to
some graphs in
$\mathcal{L}_n$ (when in fact all probabilities should be positive) and
the latter case renders statistical inference based on such an estimator ill-posed.

In order to study both issues, we have obtained
numerically all the possible
maximum likelihood estimates under the constraints of exchangeability and
dissociatedness for all the realizations  of one network on four nodes.
We have carried out the calculations in {\tt Mathematica} using the built in
optimization method. \cite{drt08} propose a general algorithm for computing the MLE of the M\"{o}bius
    parameters of marginal models for binary data that could
in principle be used for our problem. While such algorithm is more
efficient and presumably faster than the brute force optimization, it requires strictly positive counts, a
condition that is never satisfied when the data take the form of a single observed
network.

When $n = 4$, there are $11$ isomorphism classes,
shown below in Figure \ref{fig:4nodes} as unlabeled graphs, along with their respective sizes.

\begin{figure}[!htb]
\centering
\begin{tikzpicture}[node distance = 5mm and 5mm, minimum width = 1mm]
    \begin{scope}
      \tikzstyle{every node} = [shape = circle,
      font = \small,
      minimum height = 1mm,
      inner sep = 2pt,
      draw = black,
      fill = black,
      anchor = center],
      text centered]
      \node(i) at (0,0) {};
      \node(j) [right = of i] {};
      \node(l) [above = of j] {};
			\node(k) [above = of i] {};
      \node(i1) [above = 15mm of i] {};
      \node(j1) [right = of i1] {};
      \node(l1) [above = of j1] {};
			\node(k1) [above = of i1] {};
			
			\node(i2) [right = 15mm of i] {};
      \node(j2) [right = of i2] {};
      \node(l2) [above = of j2] {};
			\node(k2) [above = of i2] {};
			
			\node(i3) [above = 15mm of i2] {};
      \node(j3) [right = of i3] {};
      \node(l3) [above = of j3] {};
			\node(k3) [above = of i3] {};
			
			\node(i4) [left = 15mm of i] {};
      \node(j4) [right = of i4] {};
      \node(l4) [above = of j4] {};
			\node(k4) [above = of i4] {};
			
			\node(i5) [above = 15mm of i4] {};
      \node(j5) [right = of i5] {};
      \node(l5) [above = of j5] {};
			\node(k5) [above = of i5] {};

			\node(i6) [right = 15mm of i2] {};
      \node(j6) [right = of i6] {};
      \node(l6) [above = of j6] {};
			\node(k6) [above = of i6] {};

			\node(i7) [above = 15mm of i6] {};
      \node(j7) [right = of i7] {};
      \node(l7) [above = of j7] {};
			\node(k7) [above = of i7] {};

			\node(i8) [left = 15mm of i4] {};
      \node(j8) [right = of i8] {};
      \node(l8) [above = of j8] {};
			\node(k8) [above = of i8] {};

			\node(i9) [above = 15mm of i8] {};
      \node(j9) [right = of i9] {};
      \node(l9) [above = of j9] {};
			\node(k9) [above = of i9] {};

			\node(i10) [right = 15mm of i7] {};
      \node(j10) [right = of i10] {};
      \node(l10) [above = of j10] {};
			\node(k10) [above = of i10] {};		
    \end{scope}
		
    \begin{scope}
    \tikzstyle{every node} = [node distance = 4mm and 4mm, minimum width = 1mm,
    font= \footnotesize,
      anchor = center,
      text centered]
\node(a) [below left = 1mm and 0mm of j]{$\times 12$};
\node(b) [below left = 1mm and 0mm of j1]{$\times 12$};
\node(c) [below left = 1mm and 0mm of j2]{$\times 6$};
\node(d) [below left = 1mm and 0mm of j3]{$\times 3$};
\node(e) [below left = 1mm and 0mm of j4]{$\times 3$};
\node(f) [below left = 1mm and 0mm of j5]{$\times 6$};
\node(g) [below left = 1mm and 0mm of j6]{$\times 1$};
\node(h) [below left = 1mm and 0mm of j7]{$\times 4$};
\node(p) [below left = 1mm and 0mm of j8]{$\times 4$};
\node(q) [below left = 1mm and 0mm of j9]{$\times 1$};
\node(r) [below left = 1mm and 0mm of j10]{$\times 12$};
\end{scope}
    \begin{scope}
      \draw (i) -- (j);
			\draw (i) -- (k);
			\draw (j) -- (k);
			\draw (l) -- (k);
			
			\draw (i1) -- (j1);
			\draw (i1) -- (k1);
			
			\draw (i2) -- (j2);
			\draw (i2) -- (k2);
			\draw (j2) -- (l2);
			\draw (l2) -- (k2);
			\draw (j2) -- (k2);
			
			\draw (i3) -- (j3);
			\draw (l3) -- (k3);
			
			\draw (i4) -- (j4);
			\draw (i4) -- (k4);
			\draw (j4) -- (l4);
			\draw (l4) -- (k4);
			
			\draw (i5) -- (j5);

			\draw (i7) -- (j7);
			\draw (i7) -- (k7);
			\draw (k7) -- (j7);

			\draw (i10) -- (j10);
			\draw (i10) -- (k10);
			\draw (k10) -- (l10);

			\draw (i8) -- (j8);
			\draw (i8) -- (k8);
			\draw (i8) -- (l8);

			\draw (i6) -- (j6);
			\draw (i6) -- (k6);
			\draw (j6) -- (l6);
			\draw (l6) -- (k6);
			\draw (j6) -- (k6);
			\draw (i6) -- (l6);
    \end{scope}

    \end{tikzpicture}
		\caption{{\footnotesize All non-isomorphic graphs $U \in
		    \mathcal{U}_4$, along with the size $m$ of the isomorphism
		    class each represents, denoted by
		$\times m$.}\label{fig:4nodes}}
		\end{figure}
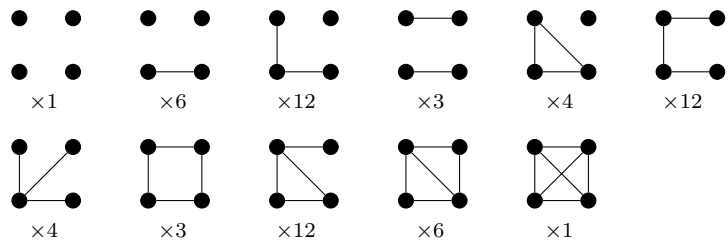

Table \ref{tab:zu} and \ref{tab:pu} show the maximum likelihood estimates of the
M\"{o}bius parameter and of the actual probabilities, respectively. An empty
entry in the table signifies a value of zero.
It is apparent from Table \ref{tab:pu} that  all the estimates
contain zero coordinates, a fact that  implies that, with only one observed
network, all the maximum likelihood estimates lie on the boundary of the parameter
space. Furthermore, the MLE is not unique: it can be seen from Tables
\ref{tab:zu} and \ref{tab:pu} that when the
observed network consists of two parallel edges, or is a path or a cycle then the likelihood is maximized along
line
segments on the boundary of both the simplex and the M\"{o}bius simplex. (See
also Example 7 in \cite{lauritzen:rinaldo:sadeghi:18}).
Finally, direct calculations reveal that, with the exception of
    the point masses at the empty and complete graphs, none of the maximum likelihood estimates
of the probability distributions extend to exchangeable distributions over larger
networks.

The fact that there
are zeros in the MLEs of \Cref{tab:pu} means that it is relatively easy to check, for each
case, that no exchangeable distribution on $5$-node graphs can marginalize to
that MLE. Consider for example the second row, corresponding to observing a
$4$-node graph with only one edge. The MLE is a mixture of a point mass at the
complete graph and of the uniform distribution over graphs isomorphic to the
observed one. In order for the MLE to be the marginal of some exchangeable
distribution on $5$-node graphs, that distribution must in turn be a mixture
of uniform distributions over isomorphic $5$-node graphs (and a point mass on the
complete graph on $5$ nodes) such that the removal of any one node will either
be a $4$-node graph with one edge or a complete graph. Such distribution does
not exist (because there does not exist any $5$-node unlabeled graph such that
removing any one node will produce as an induced subgraph a $4$-node graph with only
one edge). Other cases can be checked by similar arguments.

\begin{landscape}
\begin{table}
\caption{\label{tab:zu}\small Maximum likelihood estimates of the M\"{o}bius parameters
 for all possible samples of size
one from $\mathcal{L}_{4}$.  Along the rows of
the table we report only the isomorphic classes in
$\mathcal{G}_{4}$, each represented as an undirected graph in $\mathcal{U}_4$. Indeed, by
    exchangeability, isomorphic graphs in $\mathcal{G}_4$ yield the
same maximum likelihood estimate. The columns of the
table are indexed by all $U \in \mathcal{U}_4$ without isolated nodes. Thus the entry $(U',U)$
represents the maximum likelihood estimate of $z(F)$, where $U = [F]$ if the observed network is
a graph in the class represented by $U'$.
 An empty cell is equivalent to $0$. In addition, $-(3\sqrt{2}/4)+1\leq b,d\leq (3\sqrt{2}/4)-1$ and $-1/32\leq c\leq 1/32$.}
\centering
	\fbox{%
	\begin{tabular}{|c|c|c|c|c|c|c|c|c|c|c|}
  \hline
  $x\backslash U$  & $\vcenter{\hbox{\includegraphics[scale=0.05]{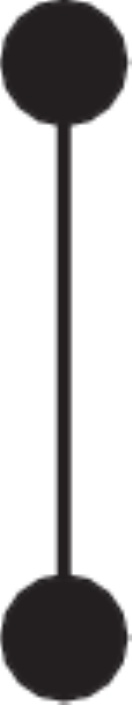}}}$ & $\vcenter{\hbox{\includegraphics[scale=0.05]{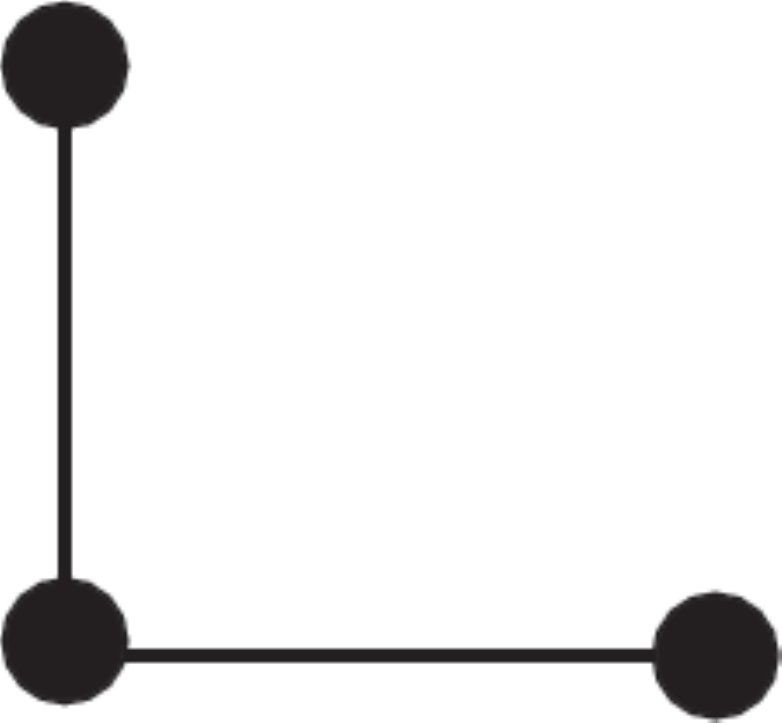}}}$ & $\vcenter{\hbox{\includegraphics[scale=0.05]{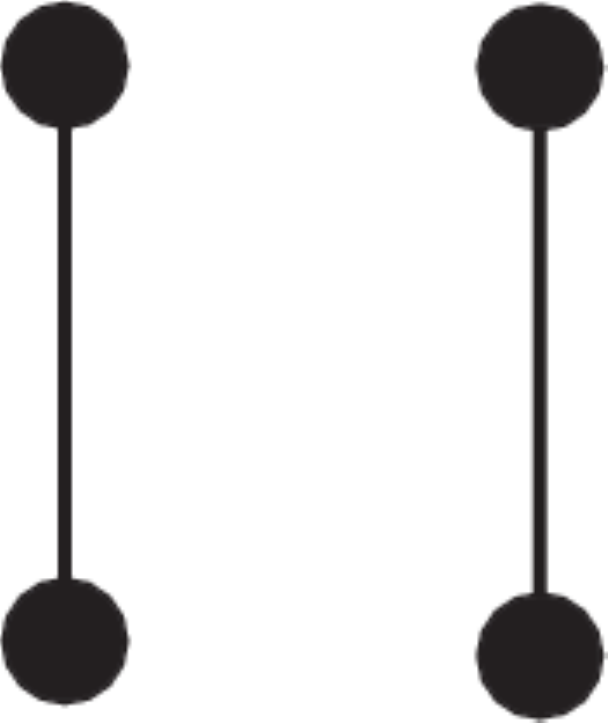}}}$ & $\vcenter{\hbox{\includegraphics[scale=0.05]{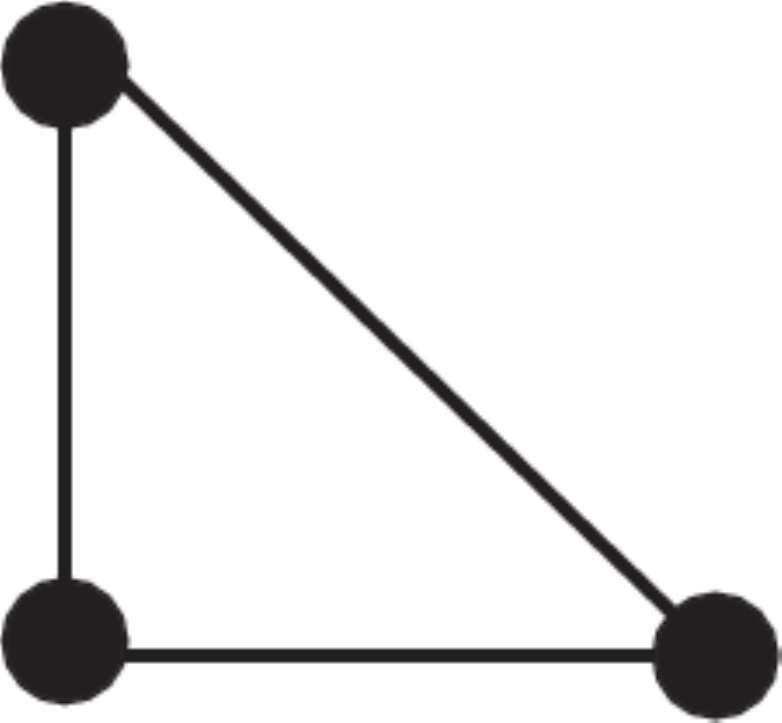}}}$ & $\vcenter{\hbox{\includegraphics[scale=0.05]{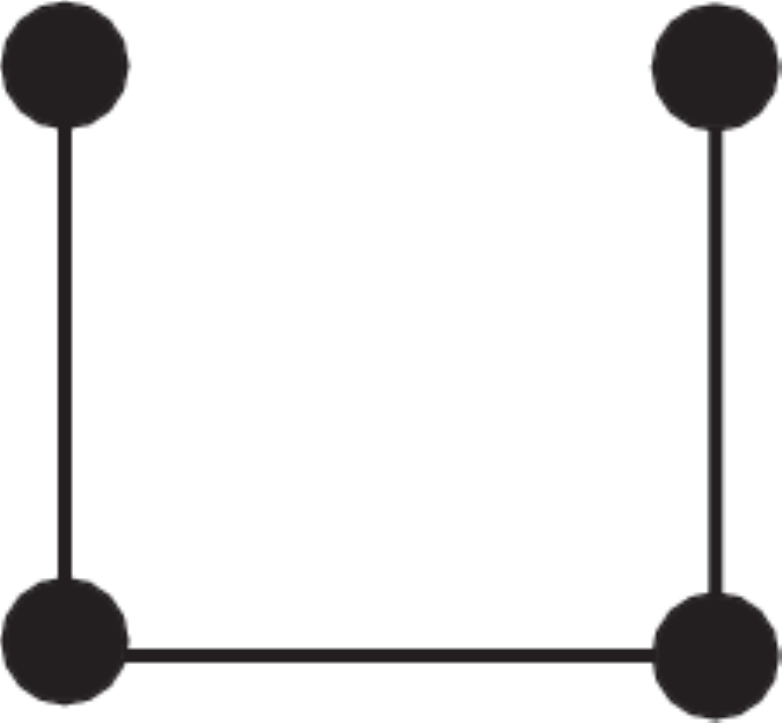}}}$ & $\vcenter{\hbox{\includegraphics[scale=0.05]{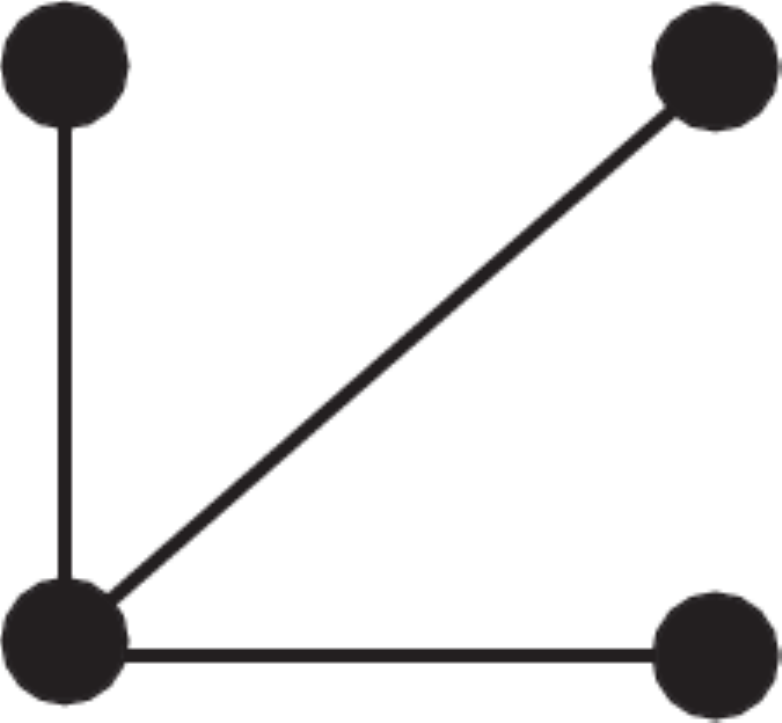}}}$ & $\vcenter{\hbox{\includegraphics[scale=0.05]{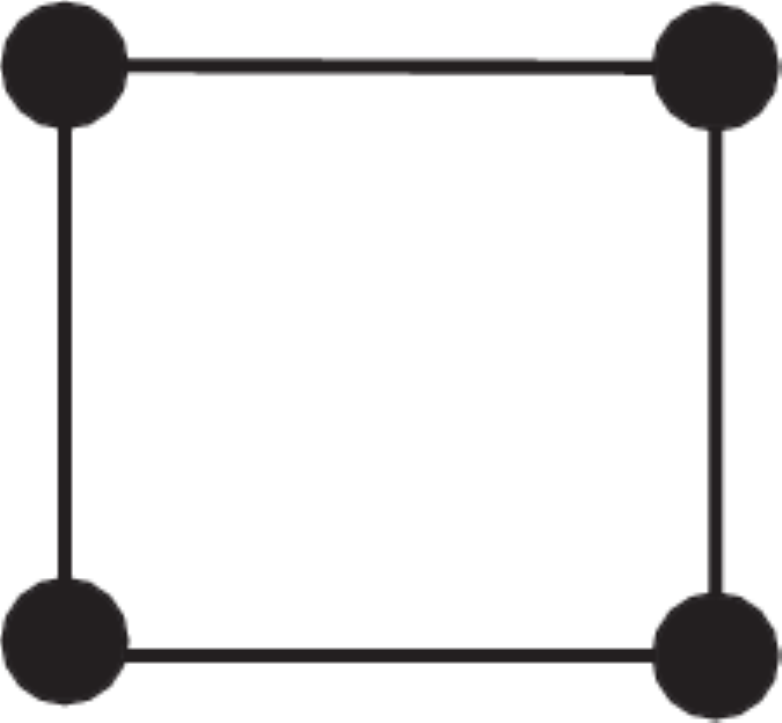}}}$ & $\vcenter{\hbox{\includegraphics[scale=0.05]{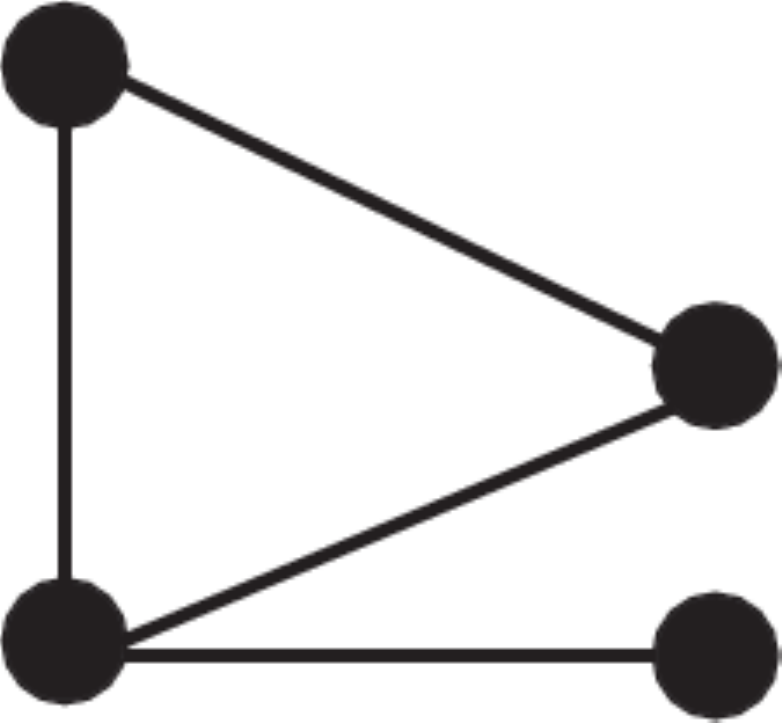}}}$ & $\vcenter{\hbox{\includegraphics[scale=0.05]{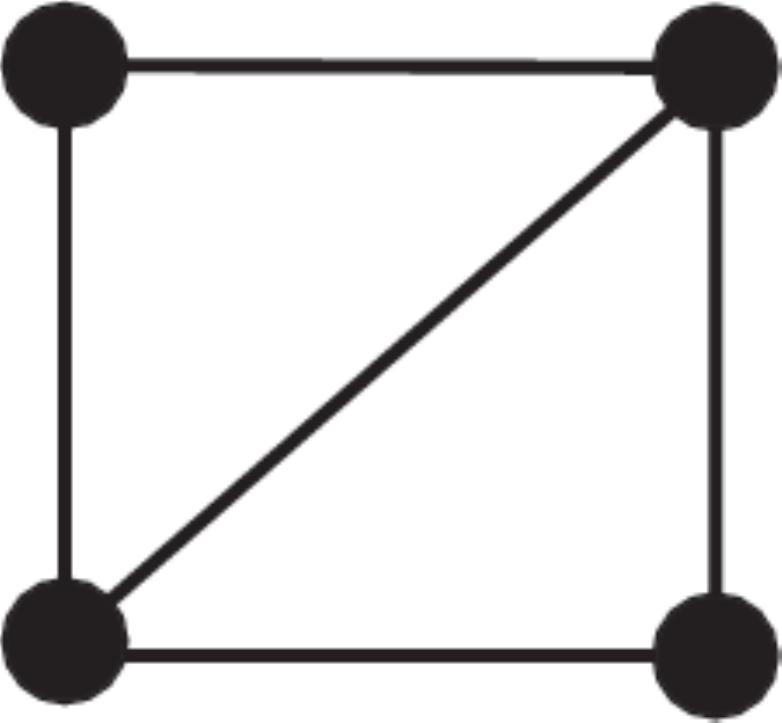}}}$  & $\vcenter{\hbox{\includegraphics[scale=0.05]{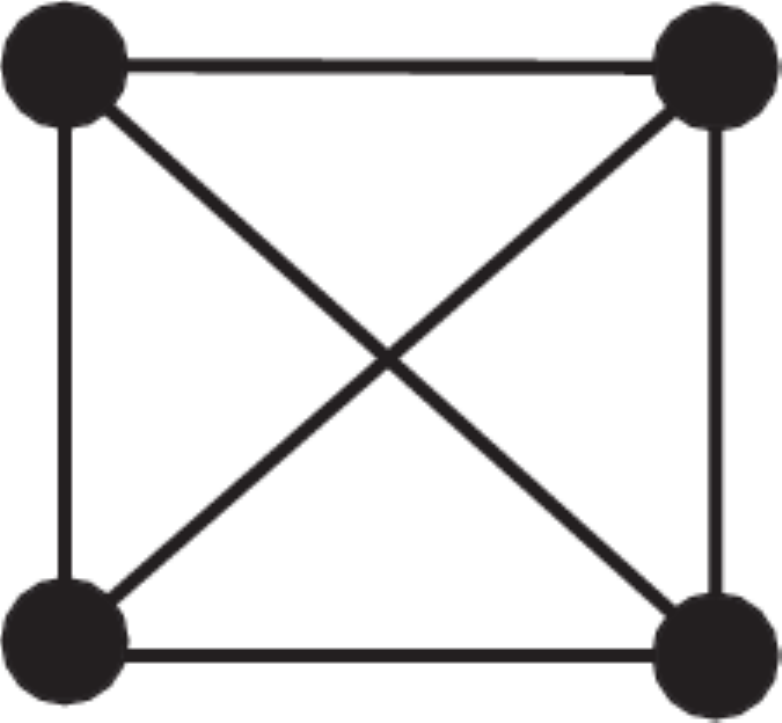}}}$ \\ \hline 
  $\vcenter{\hbox{\includegraphics[scale=0.05]{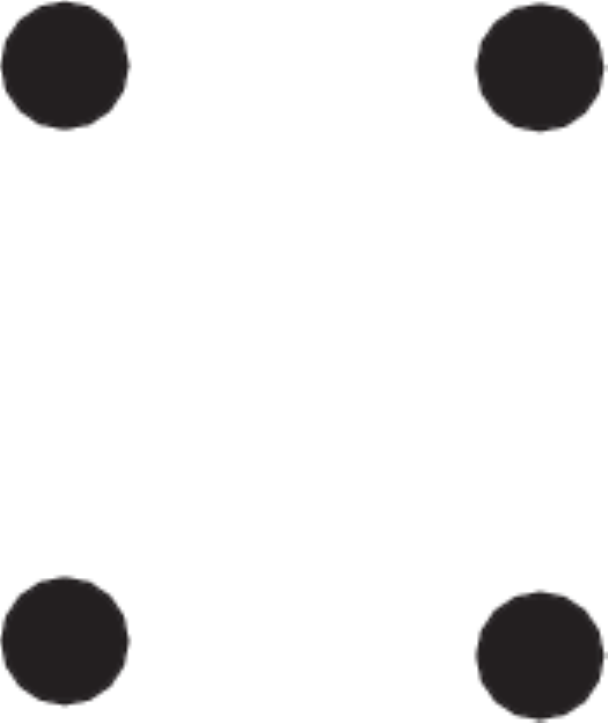}}}$  &  &  &  &  &  &  &  &  &  &  \\
  $\vcenter{\hbox{\includegraphics[scale=0.05]{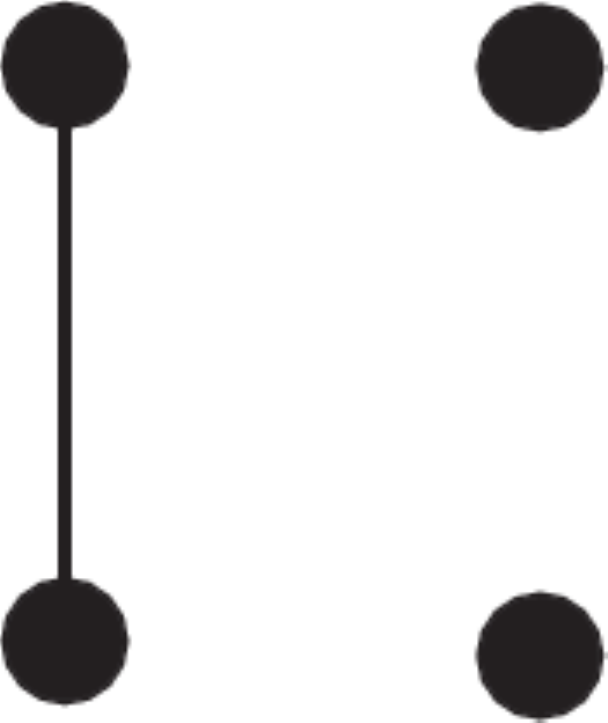}}}$  & $1/5$ & $1/25$ & $1/25$ & $1/25$ & $1/25$ & $1/25$ & $1/25$ & $1/25$ & $1/25$ & $1/25$ \\
  $\vcenter{\hbox{\includegraphics[scale=0.05]{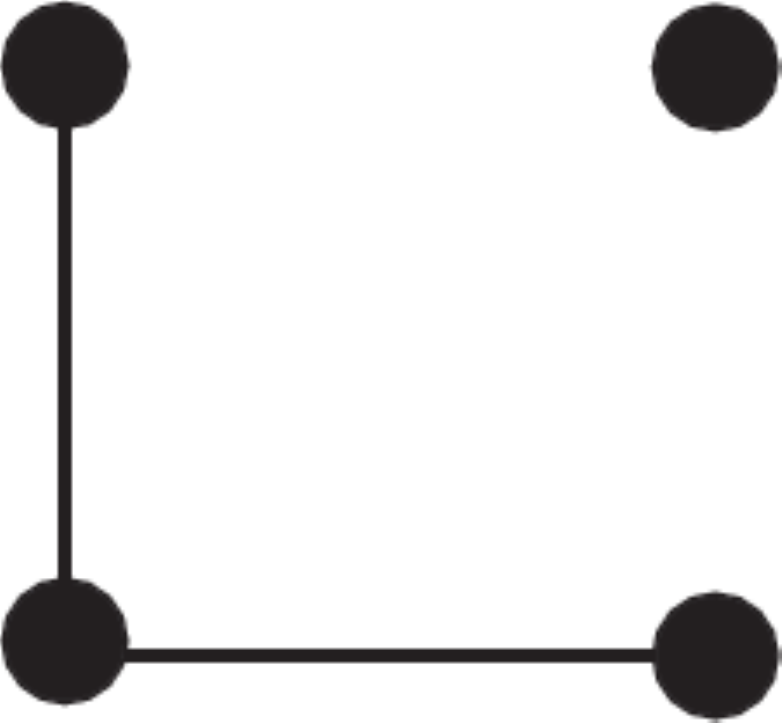}}}$  & $1/2$ & $5/16$ & $1/4$ & $1/4$ & $1/4$ & $1/4$ & $1/4$ & $1/4$ & $1/4$ & $1/4$ \\
  $\vcenter{\hbox{\includegraphics[scale=0.05]{figureex2-3.pdf}}}$  &  $\sqrt{2}-1$ & $(3\sqrt{2}/2)-2$ & $3-2\sqrt{2}$ & $(3\sqrt{2}/4)-1-b$ &  & $(3\sqrt{2}/4)-1+b$ &  &  &  &  \\
  $\vcenter{\hbox{\includegraphics[scale=0.05]{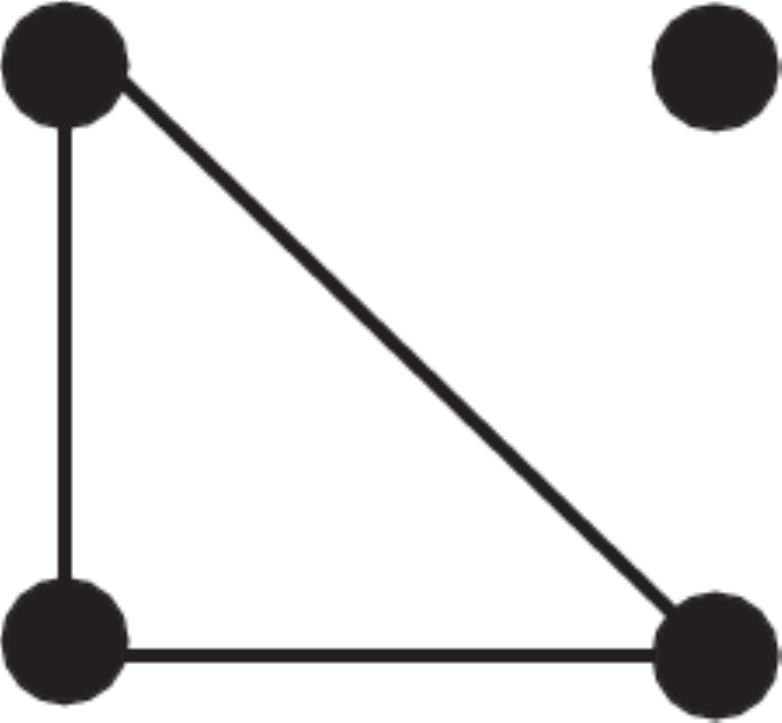}}}$  & $1/2$ & $1/4$ & $1/4$ & $1/8$ & $1/8$ &  & $1/8$ &  &  &  \\
  $\vcenter{\hbox{\includegraphics[scale=0.05]{figureex2-6.pdf}}}$  & $1/2$ & $3/16$ & $1/4$ & $1/32-c$ & $1/16$ & $1/32+c$ &  &  &  &  \\
  $\vcenter{\hbox{\includegraphics[scale=0.05]{figureex2-5.pdf}}}$  & $1/2$ & $1/4$ & $1/4$ &  & $1/8$ & $1/8$ & $1/8$ &  &  &  \\
  $\vcenter{\hbox{\includegraphics[scale=0.05]{figureex6-1.pdf}}}$  & $2-\sqrt{2}$ & $1-\sqrt{2}/2$ & $6-4\sqrt{2}$ & $(3\sqrt{2}/4)-1-d$ & $3-2\sqrt{2}$ & $(3\sqrt{2}/4)-1+d$ & $3-2\sqrt{2}$ &  &  &  \\
  $\vcenter{\hbox{\includegraphics[scale=0.05]{figureex2-7.pdf}}}$  & $1/2$ & $5/16$ & $1/4$ & $3/16$ & $1/8$ & $3/16$ &  & $1/16$ &  &  \\
  $\vcenter{\hbox{\includegraphics[scale=0.05]{figureex2-8.pdf}}}$   & $4/5$ & $16/25$ & $16/25$ & $12/25$ & $12/25$ & $12/25$ & $8/25$ & $8/25$ & $4/25$ &  \\
  $\vcenter{\hbox{\includegraphics[scale=0.05]{figureex2-9.pdf}}}$  & $1$ & $1$ & $1$ & $1$ & $1$ & $1$ & $1$ & $1$ & $1$ & $1$ \\
  \hline
\end{tabular}}

\end{table}

\begin{table}
\caption{\label{tab:pu} \small Maximum likelihood estimates of the probability
parameters for all possible samples of size
one from $\mathcal{L}_{4}$.
Along the rows and columns of
the Table we report only the isomorphic classes in
$\mathcal{G}_{4}$, each represented as an undirected graph in $\mathcal{U}_4$.
Thus, the entry corresponding to the pair $(U',U)$ is the
maximum likelihood estimate of the probability of observing a graph in the
isomorphism class represented by $U$  if the observed network is a graph in the class
represented by $U'$.
 An empty cell is equivalent to $0$. In addition, $-(3\sqrt{2}/4)+1\leq b,d\leq (3\sqrt{2}/4)-1$ and $-1/32\leq c\leq 1/32$.}
\centering
\fbox{%
\begin{tabular}{|c|c|c|c|c|c|c|c|c|c|c|c|}
  \hline
    {$x\backslash U$} & $\vcenter{\hbox{\includegraphics[scale=0.05]{figureex2-04.pdf}}}$ & $\vcenter{\hbox{\includegraphics[scale=0.05]{figureex2-14.pdf}}}$ & $\vcenter{\hbox{\includegraphics[scale=0.05]{figureex2-24.pdf}}}$ & $\vcenter{\hbox{\includegraphics[scale=0.05]{figureex2-3.pdf}}}$ & $\vcenter{\hbox{\includegraphics[scale=0.05]{figureex2-44.pdf}}}$ & $\vcenter{\hbox{\includegraphics[scale=0.05]{figureex2-6.pdf}}}$ & $\vcenter{\hbox{\includegraphics[scale=0.05]{figureex2-5.pdf}}}$ & $\vcenter{\hbox{\includegraphics[scale=0.05]{figureex6-1.pdf}}}$ & $\vcenter{\hbox{\includegraphics[scale=0.05]{figureex2-7.pdf}}}$ & $\vcenter{\hbox{\includegraphics[scale=0.05]{figureex2-8.pdf}}}$  & $\vcenter{\hbox{\includegraphics[scale=0.05]{figureex2-9.pdf}}}$ \\ \hline
  $\vcenter{\hbox{\includegraphics[scale=0.05]{figureex2-04.pdf}}}$  & $1$ &  &  &  &  &  &  &  &  &  &  \\
  $\vcenter{\hbox{\includegraphics[scale=0.05]{figureex2-14.pdf}}}$  & & $4/25$ &  &  &  &  &  &  &  &  & $1/25$ \\
  $\vcenter{\hbox{\includegraphics[scale=0.05]{figureex2-24.pdf}}}$  & &  & $1/16$ &  &  &  &  &  &  &  & $1/4$ \\
  $\vcenter{\hbox{\includegraphics[scale=0.05]{figureex2-3.pdf}}}$  & &  &  &
  $3-2\sqrt{2}$ &
  $(3\sqrt{2}/4)-1-b$ & &
  $(3\sqrt{2}/1)-1+b$  &  &  &  &  \\
  $\vcenter{\hbox{\includegraphics[scale=0.05]{figureex2-44.pdf}}}$  & $1/8$ &  &  &  & $1/8$ &  &  & $1/8$ &  &  &  \\
  $\vcenter{\hbox{\includegraphics[scale=0.05]{figureex2-6.pdf}}}$  & &  &  &  & $1/32-c$ & $1/16$ & $1/32+c$ &  &  &  &  \\
  $\vcenter{\hbox{\includegraphics[scale=0.05]{figureex2-5.pdf}}}$  & $1/8$ &  &  &  &  &  & $1/8$ & $1/8$ &  &  &  \\
  $\vcenter{\hbox{\includegraphics[scale=0.05]{figureex6-1.pdf}}}$  & &  &  &  & $(3\sqrt{2}/4)-1-d$  &   & $(3\sqrt{2}/4)-1+d$ & $3-2\sqrt{2}$ &  &  &  \\
  $\vcenter{\hbox{\includegraphics[scale=0.05]{figureex2-7.pdf}}}$  & $1/4$ &  &  &  &  &  &  &  & $1/16$ &  &  \\
  $\vcenter{\hbox{\includegraphics[scale=0.05]{figureex2-8.pdf}}}$   & $1/25$ &  &  &  &  &  &  &  &  & $4/25$ &  \\
  $\vcenter{\hbox{\includegraphics[scale=0.05]{figureex2-9.pdf}}}$  & &  &  &  &  &  &  &  &  &  & $1$ \\
  \hline
\end{tabular}}
\end{table}
\end{landscape}

 \section{Conclusions}

 It is worth commenting on the difference between Theorem
 \ref{thm:main} and \ref{thm:definetti} and
 analogous results for finitely exchangeable and exchangeable random binary sequences; see, e.g.
 \cite{diaconis:77,diaconis:freedman:81,ker06}. First, the $n$-dimensional marginal of any exchangeable binary sequence can
 be described geometrically  as a uniquely determined point in the convex hull of the intersection of the one-dimensional variety
 corresponding to the surface of independence inside the $(2^n - 1)$-dimensional
 simplex
  with the $n$ dimensional affine subset of finitely exchangeable distributions. In our
  setting the manifold $\mathcal{D}_n$ of dissociated distributions (actually,
  as we saw, a certain non-trivial subset of it) plays
  an analogous role, though in the M\"{o}bius parametrization. However, it is clear that
  $\mathcal{D}_n$ is much
 more complex, and, unlike the surface of independence, which has fixed
 dimension $1$, its dimension increases
 with $n$. 
Table \ref{tab:1} provides the dimension of $\mathcal{E}_n$ and of
$\mathcal{D}_n$ (which are the number of unlabeled graphs minus $1$ and the number of connected
unlabeled  graphs, respectively) for all nodes of size
up to $11$. As it can be seen and is also simple to show, the ratio of the dimension of $\mathcal{D}_n$
 over that of $\mathcal{E}_n$ converges to $1$ very rapidly as $n$ grows.  
 The other striking difference is the fact that not all points on the manifold
 of dissociated distributions correspond to extremal exchangeable
 distributions. This is in contrast with the sequence case, in which every
 point
 on the surface of independence corresponds to an extremal exchangeable
 distribution, for each $n$. Thus, exchangeability in graphs (a special
 case of exchangeability for binary
 2-dimensional array) is considerably more subtle and complicated than the
 sequence case.

\begin{table}
\caption{\label{tab:1}{\footnotesize Dimension of the sets $\mathcal{E}_n$ and
$\mathcal{D}_n$   a
	function of the number of nodes $n$. The numbers are sourced from
	OEIS Foundation Inc.
    (2011), The On-Line Encyclopedia of Integer Sequences,
    see {\tt http://oeis.org/A000088} and {\tt http://oeis.org/A001349}}.
	}
	\centering
	\fbox{%
	\begin{tabular}{l|ccccccccc}
	    	\hline \footnotesize
		& \multicolumn{9}{c}{$n$}\\
		\hline
	    	 &  3 &4 &5 &6 &7 &8 &9
	    	&10 & 11  \\
				\hline
				$\mathrm{dim}(\mathcal{E}_n)$ & 3 & 10 & 33 & 155 & 1,043 & 12,345  & 274,667 & 12,005,167 & 1,018,997,863 \\
 $\mathrm{dim}(\mathcal{D}_n)$ & 3 & 9 & 30 & 142 & 995 & 12,112 & 273,192 & 11,989,763 & 1,018,690,328 \\
\hline
	    \end{tabular}}
	\end{table}

	\section*{Acknowledgments}
	The authors have benefited from precise and constructive comments from two anonymous referees. 
Alessandro Rinaldo and Kayvan Sadeghi were partially supported by AFOSR grant
FA9550-14-1-014.

\bibliography{jasbib}

\end{document}